\documentclass[a4paper,11pt]{article}
\usepackage[utf8]{inputenc} 
\usepackage[margin=1in,footskip=0.25in, a4paper]{geometry}

\usepackage[english]{babel}
\usepackage{amsmath,amssymb, amsthm, enumerate}
\usepackage{bbm, dsfont}
\usepackage{tikz}
\usepackage[inline]{enumitem}
 
\usepackage{natbib}

\usepackage[utf8]{inputenc}
\usepackage{hyperref}
\usepackage{graphicx}
\usepackage[font=small]{caption}
\usepackage{subcaption}

\newtheorem{theorem}{Theorem}
\newtheorem{proposition}{Proposition}
\newtheorem{lemma}{Lemma}
\newtheorem{corollary}{Corollary}
\newtheorem{definition}{Definition}
\newtheorem{example}{Example}
\newtheorem{remark}{Remark}
\usepackage{authblk}

\newtheorem{algorithm}{Algorithm}

\def\g{\boldsymbol}
\def\ci{\perp\!\!\!\perp}
\def\perpe{\perp_e}
\DeclareMathOperator{\argmin}{arg\,min}

\DeclareMathOperator{\diag}{diag}
\DeclareMathOperator{\ph}{ph}

\DeclareMathOperator{\AIC}{AIC}
\DeclareMathOperator{\MST}{mst}
\DeclareMathOperator{\cens}{cens}

\newcommand{\einsfun}{\g 1} 

\title{Graphical Models for Extremes}

\author[1]{Sebastian Engelke}
\author[2]{Adrien S.~Hitz}

\affil[1]{Research Center for Statistics, University of Geneva, Boulevard du Pont d’Arve 40, 1205 Geneva, Switzerland.}

\affil[2]{Department of Statistics, University of Oxford, 24-29 St Giles', Oxford OX1 3LB, UK and 
  Materialize.X, Enterprise Lab, Imperial College London, London SW7 2AZ, UK.}

\date{}

\begin{document}
\maketitle

\begin{abstract}
    Conditional independence, graphical models and sparsity are key notions for parsimonious statistical models and for understanding the structural relationships in the data. The theory of multivariate and spatial extremes describes the risk of rare events through asymptotically justified limit models such as max-stable and multivariate Pareto distributions. Statistical modelling in this field has been limited to moderate dimensions so far, partly owing to complicated likelihoods and a lack of understanding of the underlying probabilistic structures. We introduce a general theory of conditional independence for multivariate Pareto distributions that allows the definition of graphical models and sparsity for extremes. A Hammersley--Clifford theorem links this new notion to the factorization of densities of extreme value models on graphs. For the popular class of H\"usler--Reiss distributions we show that, similarly to the Gaussian case, the sparsity pattern of a general extremal graphical model can be read off from suitable inverse covariance matrices. New parametric models can be built in a modular way and statistical inference can be simplified to lower-dimensional marginals. We discuss learning of minimum spanning trees and model selection for extremal graph structures, and illustrate their use with an application to flood risk assessment on the Danube river.
\end{abstract}

\section{Introduction}\label{sec:intro}

Evaluation of the risk related to heat waves, extreme flooding, financial crises, or other
rare events requires the quantification of their small occurrence probabilities.
Empirical estimates are unreliable since the regions of interest are in the
tail of the distribution and typically contain few or no data points.
Extreme value theory provides the theoretical foundation for extrapolations
to the distributional tail of a $d$-dimensional random vector $\g X$.
The univariate case $d=1$ is well-studied and the generalized extreme value and Pareto distributions are widely applied in areas such as hydrology \citep{kat2002}, climate science \citep{min2011} and finance \citep{mcn2015}; see also \cite{emb1997} and \cite{ber2004}.

In the multivariate setting, $d\geq 2$, the result of the extrapolation strongly depends on the strength of extremal dependence between the components of $\g X$. 
Most current statistical models assume multivariate regular variation for $\g X$ \citep{res2008} since
this entails mathematically elegant descriptions of the asymptotic tail distribution. Similar to the univariate setting, two different but closely related approaches exist.
Max-stable distributions arise as limits of normalised maxima of independent copies of $\g X$
and have been extensively studied and applied in multivariate and spatial risk problems \citep[cf.,][]{deh1984, seg2010, dav2012b}.
On the other hand, multivariate Pareto distributions describe the random vector $\g X$
conditioned on the event that at least one component exceeds a high threshold; see \cite{roo2006}, \cite{roo2018} and \cite{kir2018} for their construction, stability properties and statistical inference.

A drawback of the current multivariate models is their limitation to rather moderate dimensions $d$, and the construction of tractable parametric models in higher dimensions is challenging, both for max-stable and multivariate Pareto distributions.
Sparse multivariate models require the notion of conditional independence \citep{daw1979}, which is not easy to define for tail distributions. In fact, \cite{papastathopoulos2016conditional} show that if $(Z_1, Z_2, Z_3)$ is a max-stable random vector with positive continuous density, then the conditional independence of $Z_1 \ci Z_3 \mid Z_2$ already implies the independence $Z_1\ci Z_3$; see also \cite{dom2014}. Meaningful conditional independence structures can thus only be obtained for max-stable distributions with discrete spectral measure \citep{gis2018}. Since these models do not admit densities, this excludes most of the currently used parametric families.

In this work we take the perspective of threshold exceedances and introduce a new notion of
conditional independence for a multivariate Pareto distribution $\g Y = (Y_1,\dots, Y_d)$,
which we denote by $\perpe$ to stress that it is designed for extremes.
It is different from classical conditional independence since the support of $\g Y$
is not a product space, but the homogeneity property of $\g Y$ can be used to show that it is well-defined.
Conditional independence is tightly linked to graphical models. For an undirected graph $\mathcal G = (V,E)$
with nodes $V=\{1,\dots, d\}$ and edge set $E$, we say that $\g Y$ is an extremal graphical model if it satisfies the pairwise Markov property
\begin{align}\label{egm}
  Y_i \perpe Y_j \mid \g Y_{\setminus \{i,j\}}, \quad  (i,j)\notin E.
\end{align}
The main advantage of conditional independence and graphical models is that they imply a simple
probabilistic structure and possibly sparse patterns in multivariate random vectors \citep{Lauritzen, wainwright2008graphical}.
For extremal graphical models on decomposable graphs, we prove a Hammersley--Clifford type theorem stating that~\eqref{egm} is equivalent to the factorization of the density $f_{\g Y}$ of $\g Y$
into lower-dimensional marginals. This underlines that our notion of conditional independence is in fact natural for multivariate Pareto distributions.

Applications of this result are manifold. From a probabilistic perspective, we analyse
models in the literature regarding their graphical properties in the sense of our definition~\eqref{egm}.
Extremal graphical models whose underlying graph is a tree have a particularly
simple multiplicative stochastic representation in terms of extremal functions, a notion
that is known from the simulation of max-stable processes \citep{dom2016}.
In multivariate extremes, one may argue that the family of H\"usler--Reiss distributions \citep{Husler1989} takes a
similar role as Gaussian distributions in the non-extreme world.
Instead of covariance matrices, they are parameterized by a variogram matrix $\Gamma$.
We show that the extremal graphical structure of a H\"usler--Reiss distribution can be identified by zero patterns on matrices derived from $\Gamma$.

Extremal graphical models enable the construction of parsimonious models for multivariate Pareto
distributions $\g Y$, which further enjoy the advantage of interpretability in terms of
the underlying graph. Thanks to the factorization of the densities, statistical inference
can be efficiently carried out on lower-dimensional marginals. For decomposable graphs with singleton separator sets, so-called block graphs, this allows the use of multivariate Pareto models in
fairly high dimensions. In many cases the underlying graphical structure is unknown and has to be learned from data.  
We discuss how a maximum likelihood tree can be obtained using standard algorithms by
\cite{kruskal1956shortest} or \cite{pri1957},
and how the best model can be selected among different extremal graphical models.

There is previous work on the construction of parsimonious extreme value models. A large body of literature
studies spatial max-stable random fields \citep{sch2002, kab2009, opi2013}. Such models have small parameter dimension but they rely on strong assumptions on stationarity and cannot be applied to multivariate, non-spatial data without information on an underlying space.
Other approaches include constructions through factor copulas \citep{lee2018},
ensembles of trees combining bivariate copulas \citep{yu2017}, graphical models
for large censored observations \citep{Hitz2015} and eigendecompositions \citep{coo2018}.
Closely related to our concept of conditional independence is the work of
\cite{coles1991modelling} and \cite{smi1997} who propose a Markov chain model where all bivariate marginals are extreme value distributions. This can be seen as a special case of our approach when the graph has the simple structure of a chain. Similar limiting objects also arise as the tail chains in the theory of extremes of stationary Markov chains with regularly varying marginals \citep{smi1992,bas2009,jan2014}. This theory has recently been extended to regularly varying Markov trees \citep{seg2019}. 
\cite{gis2018} and \cite{gis2018a} study the causal structure of
directed acyclic graphs for max-linear models, and they develop methods for
model identification based on tail dependence coefficients. Their work is in some sense
complementary to ours, since their models do not possess densities whereas we will
explicitly assume the existence of densities.

To the best of our knowledge, our work is the first principled attempt to define conditional
independence for general multivariate extreme value models that naturally extends to the
factorization of densities, sparsity and graphical models.
Section \ref{sec:background} introduces background on extreme value theory and graphical models needed throughout the paper. The new notion of conditional independence is defined in
Section \ref{sec:CI} and equivalent properties are derived. Section~\ref{graph_models_extremes} contains the main
probabilistic results on extremal graphical models, the representation of trees and
the characterization for H\"usler--Reiss distributions.
Statistical models on block graphs and their estimation, simulation and model selection are discussed in Section \ref{sec:stat}. In these graphical
models the dependence is modeled directly between lower-dimensional subsets of variables, whereas
the global dependence is implicitly implied by the conditional independence structure of the graph.
There are many potential applications of extremal graphical models. In Section~\ref{sec:appl}, we illustrate
the advantages of this structured approach compared to classical extreme value models
on a data set related to flooding on a river network in the upper Danube basin \citep[cf.,][]{asadi2015extremes}. The interpretation of the graphical structures
obtained in this application is particularly interesting since there is a seemingly natural
underlying tree associated to the flow-connections. Our conditional independence is formulated for multivariate Pareto distributions, but
the results in this paper have implications for max-stable distributions. This point and further research directions will be addressed in the discussion in Section \ref{sed:disc}.
The Appendix contains proofs and some additional results.

An implementation for R \citep{R2019} is available in the package \texttt{graphicalExtremes} \citep{graphicalExtremes}.
The code for the simulation study and application can be found in the supplementary material. 

\section{Background}
\label{sec:background}
\subsection{Notation}
We introduce some standard notation that is used throughout the paper. Symbols in boldface such as $\g x \in \mathbb R^d$ are column vectors with components denoted by $x_i$, $i\in \{1,\dots, d\}$, and operations and relations involving such vectors are meant componentwise. The vectors $\g 0=(0,\dots,0)$ and $\g 1=(1,\dots,1)$ are used as generic vectors with suitable dimension.
Denoting the index set by $V = \{1,\dots, d\}$, for a non-empty subset $I \subset V$,
we write for the subvectors $\g x_I = (x_i)_{i\in I}$ and $\g x_{\setminus I} = (x_i)_{i\in V\setminus I}$.
Similar notation is used for random vectors $\g X = (X_i)_{i\in V}$ with values in $\mathbb R^d$. For a matrix $A = (A_{ij})_{i,j\in V}\in \mathbb R^{d\times d}$ with entries indexed by $V$, and subsets $I,J\subset V$ we let
$A_{IJ} = (A_{ij})_{i\in I, j\in J}$ denote the $|I|\times |J|$ submatrix of $A$, and we abbreviate $A_I = A_{II}$. For $\g a, \g b\in \mathbb R^d$ with $\g a \leq \g b$, a multivariate interval is denoted by $[\g a,\g b] = [a_1,b_1] \times \dots \times [a_d,b_d]$. The $\ell_p$-norm of a vector $\g x\in \mathbb R^d$ for $p\geq 1$ is $\|\g x\|_p = \left(\sum_{i \in V} |x_i|^p\right)^{1/p}$, and its $\ell_\infty$-norm
is $\|\g x\|_\infty = \max_{i \in V} |x_i|$. The density of a random vector $\g X$, if it exists,
is denoted by $f_{\g X}$. The density of the marginal $\g X_I$ for a non-empty $I\subset V$ is denoted by~$f_I$, if there is no ambiguity regarding the random vector.

\subsection{Multivariate extreme value theory}
\label{mevd}

The tail behavior of the random vector $\g X = (X_1,\dots, X_d)$ can be described through two different approaches, 
one based on componentwise maxima and the other one on threshold exceedances. We briefly discuss both approaches and the close link between them.

Let $\g X_i = (X_{i1},\dots, X_{id})$, $i=1,\dots, n$, be independent copies of $\g X$ 
and denote the componentwise maximum by $\g M_n = (M_{1n}, \dots, M_{dn} )= (\max_{i=1}^n X_{i1}, \dots, \max_{i=1}^n X_{id})$.
Under mild conditions on the marginal distribution of $X_j$ there exist sequences of normalizing constants $b_{jn}\in\mathbb R$, $a_{jn} > 0$, $j=1,\dots,d$, such that
\begin{align}
  \label{gevd}
  \lim_{n\to \infty} \mathbb P\left( \frac{ M_{jn} -b_{jn}} {a_{jn}} \leq x \right) = G_j(x) =
  \exp\left\{ - \left(1 + \xi_j x \right)_+^{-1/\xi_j}\right\},\quad x\in\mathbb R,
\end{align}
where $z_+ = \max(z,0)$, and $G_j$ is the generalized extreme value distribution whose shape parameter $\xi_j\in \mathbb R$ determines the heaviness of the tail of $X_j$; see \cite{deh2006a,emb1997} and \cite{ber2004} for details.
For analysis of the dependence structure, the marginal distributions $F_j$ of $X_j$ are typically estimated first to normalise the data by $1 / \{1- F_j(X_j)\}$ to standard Pareto distributions.
For simplicity, we assume in the sequel that the $F_j$ are known and the vector $\g X$ has been normalised to standard Pareto marginals. Joint estimation of marginals and dependence is discussed in Section \ref{estimation}.

The standardized vector $\g X$ is said to be in the max-domain of attraction of the random vector $\g Z =(Z_1,\dots, Z_d)$ if for any $\g z = (z_1,\dots, z_d)$
\begin{align}
  \label{mrv}
  \lim_{n\to \infty} \mathbb P\left\{ \max_{i=1,\dots, n} X_{i1} \leq n z_1, \dots,\max_{i=1,\dots, n} X_{id} \leq n z_d  \right\} = \mathbb P(\g Z \leq \g z).
\end{align} 
In this case, $\g Z$ is max-stable with standard Fr\'echet marginals $\mathbb P( Z_j \leq z ) = \exp(-1 / z)$, $z\geq 0$, and we may write 
\begin{align}\label{max_cdf}
  \mathbb P(\g Z\leq \g z) = \exp\left\{ - \Lambda\left(\g z\right) \right\}, \qquad \g z \in \mathcal E,
\end{align} 
where the exponent measure $\Lambda$ is a Radon measure on the cone $\mathcal E = [0,\infty)^d\setminus \{\g 0\}$, and $\Lambda\left(\g z\right)$ is shorthand for $\Lambda\left(\mathcal E \setminus [\g 0,\g z] \right)$.
If $\Lambda$ is absolutely continuous with respect to Lebesgue measure on $\mathcal E$, its Radon--Nikodym derivative, denoted by $\lambda$, has the following properties:
\begin{itemize}
  \item[(L1)]
    homogeneity of order $-(d+1)$, i.e., $\lambda(t\g y) = t^{-(d+1)} \lambda(\g y)$ for any $t>0$ and $\g y \in\mathcal E$;
  \item[(L2)]
    normalised marginals, i.e., for any $i=1,\dots, d$, 
    $$\int_{\g y\in \mathcal E: y_i > 1} \lambda(\g y) \mathrm d \g y = 1.$$
\end{itemize}
The two properties follow from the max-stability and the standard Fr\'echet marginals of $\g Z$, respectively. For a non-empty subset $I\subset \{1,\dots, d\}$, we define the marginal of $\lambda$ by
\begin{align}\label{lambda_margin}
\lambda_I(\g y_I) =  \int_{[0,\infty)^{d-|I|}} \lambda(\g y) \mathrm d \g y_{\setminus I},
\end{align}
and note that it is homogeneous of order $-(|I|+1)$. In particular, if $I=\{i\}$ for some
$i=1,\dots, d$, then $\lambda_{\{i\}}(y_i) = 1/y_i^2$ as a consequence of (L1) and (L2).
Conversely, any positive and continuous function $\lambda$ satisfying (L1) and (L2) defines a valid density of
an exponent measure $\Lambda(\g z)$ by integration over $\mathcal E \setminus [\g 0,\g z]$, $\g z\in \mathcal E$,
that satisfies similar homogeneity and normalization properties as $\lambda$. By~\eqref{max_cdf} this also defines a max-stable distribution.

Another perspective on multivariate extremes is through threshold exceedances.
By Proposition 5.17 in \cite{res2008}, the convergence in~\eqref{mrv} is equivalent to
\begin{align*}
  \lim_{u\to \infty} u \{1 -  \mathbb P( \g X \leq  u \g z)\} = \Lambda(\g z), \quad \g z \in \mathcal E.
\end{align*} 
Consequently, the multivariate distribution of the threshold exceedances of $\g X$ satisfies 
\begin{align}
  \label{mpd}
  \mathbb P(\g Y \leq \g z) &= \lim_{u\to \infty} \mathbb P\left(\g X / u \leq   \g z \mid \| \g X\|_\infty > u\right)
   =  \frac{\Lambda(\g z\wedge \g 1) - \Lambda(\g z )}{\Lambda(\g 1)}, \quad \g z \in \mathcal E.
\end{align} 
The distribution of the limiting random vector $\g Y$ is called a multivariate Pareto distribution \citep[cf.,][]{roo2006}. It is defined through the exponent measure $\Lambda$ of the max-stable distribution $\g Z$, with support on the $L$-shaped space 
$\mathcal L = \{\g x \in \mathcal E: \|\g x\|_\infty > 1\}$. We say that $\g Z$ and $\g Y$ are
associated, since their distributions mutually determine each other.

Multivariate Pareto distributions are the only possible limits in~\eqref{mpd} and, owing to the homogeneity of the exponent measure, they enjoy certain stability properties \citep[cf.,][]{roo2018}. The exponent measure $\Lambda$, and hence the distribution of $\g Y$, may place mass on some lower-dimensional faces of the space $\mathcal E$. For the remainder of this paper we exclude this case to avoid technical difficulties. 
We further assume that the distribution of $\g Y$ admits a positive and continuous density $f_{\g Y}$ on $\mathcal L$, which is 
\begin{align*}
  f_{\g Y}(\g y) = \frac{\partial^d}{\partial y_1 \dots \partial y_d} \mathbb P(\g Y \leq \g y) 
  =  
  \frac{\lambda(\g y)}{\Lambda(\g 1)}, \quad \g y \in \mathcal L ,
\end{align*} 
since $\Lambda(\g y \wedge \g 1)$ is always constant along at least one coordinate for $\g y \in \mathcal L$. The density $f_{\g Y}$ is thus proportional to the density $\lambda$ of the exponent measure $\Lambda$. By the homogeneity of $\lambda$, $f_{\g Y}$ is also homogeneous of order $-(d+1)$.
The normalization constant $\Lambda(\g 1)\in [1,d]$ is known as the $d$-variate extremal coefficient \citep[cf.,][]{sch2003}. The assumption of a positive and continuous density $f_{\g Y}$ implies that
the multivariate Pareto distributions we study are models for asymptotic extremal dependence, and all $p$-variate extremal coefficients, $1\leq p\leq d$, are strictly smaller than their upper limit $p$.

For some non-empty subset $I\subset \{1,\dots, d\}$, the subvector $\g X_I = (X_j)_{j\in I}$, properly normalised, given that its $\ell_\infty$-norm is large converge in the sense of~\eqref{mpd} to the marginal $\g Y_I = (Y_j)_{j\in I}$ of $\g Y$ defined on $\mathcal L_I = \{ \g x_I \in [0,\infty)^{|I|}\setminus \{\g 0\}: \|\g x_I\|_\infty > 1\}$ with homogeneous density of order $-(|I| +1)$ given by
\begin{align}\label{margin_pareto}
  f_I(\g y_I) = \frac{\Lambda(\g 1)}{\Lambda_I(\g 1)} \int_{[0,\infty)^{d-|I|}} f_{\g Y}(\g y) \mathrm d \g y_{\setminus I} = \frac{\lambda_I(\g y_I)}{\Lambda_I(\g 1)}, \quad \g y_I \in \mathcal L_I,
\end{align}
where $\Lambda_I$ is the exponent measure of $\g Z_I$, and $\lambda_I$ is the density of $\Lambda_I$.

\begin{example}[Logistic distribution]\label{ex_log}
  The extremal logistic distribution with parameter $\theta\in(0,1)$ induces a multivariate
  Pareto distribution with density
  \begin{align}\label{dens_log}
    f_{\g Y}(\g y) = \frac{1}{d^{\theta}} \left(y_1^{-1/\theta}+\dots+y_d^{-1/\theta}\right)^{\theta-d}\prod_{i=1}^{d-1}\left(\frac{i}{\theta}-1\right) \prod_{i=1}^{d} y_i^{-1/\theta-1}, \quad \g y \in \mathcal L.
  \end{align}
\end{example}

\begin{example}[Hüsler--Reiss distribution]\label{ex:HR}
The H\"usler--Reiss distribution \citep{Husler1989} is parameterized by a symmetric,
strictly conditionally negative definite matrix $\Gamma = \{\Gamma_{ij}\}_{1\leq i,j\leq d}$ with $\diag(\Gamma) = \g 0$ and non-negative entries, that is, $\g a^\top \Gamma \g a < 0$ for all non-zero vectors $\g a \in \mathbb R^d$ with $\sum_{i=1}^d a_i = 0$. The corresponding density of the exponent measure can be written for any $k\in\{1,\dots, d\}$ as \citep[cf.,][]{Engelke2015}
\begin{align}\label{eq:fYHR}
  \lambda(\g y)
  &= y_k^{-2}\prod_{i\neq k} y_i^{-1} \phi_{d-1}\left(\g{\tilde y}_{\setminus k}; \Sigma^{(k)}\right), \quad \g y \in \mathcal E,
\end{align}
where $\phi_p(\cdot; \Sigma)$ is the density of a centred $p$-dimensional normal distribution with covariance matrix $\Sigma$, $\g{\tilde y} = \{\log(y_i/y_k) + \Gamma_{ik}/2\}_{i=1,\dots, d}$ and 
\begin{align}\label{sigma_k}
  \Sigma^{(k)}  =\frac{1}{2} \{\Gamma_{ik}+\Gamma_{jk}-\Gamma_{ij}\}_{i,j\neq k} \in\mathbb R^{(d-1)\times (d-1)}.
\end{align}
The matrix $\Sigma^{(k)}$ is strictly positive definite; see Appendix \ref{link_vario} for details.
The representation of the density in~\eqref{eq:fYHR} seems to 
depend on the choice of $k$, but, in fact, the value of the right-hand side of this equation is independent of $k$. The H\"usler--Reiss multivariate Pareto distribution has density 
$f_{\g Y}(\g y) = \lambda(\g y) / \Lambda(\mathbf 1)$ and the
strength of dependence between the $i$th and $j$th component is parameterized by $\Gamma_{ij}$, ranging from complete dependence for $\Gamma_{ij}=0$ and independence for $\Gamma_{ij}=+\infty$. In the bivariate case $d=2$ we have
\begin{align}\label{eq:HRbiv}
  \lambda(y_1,y_2)= \frac{y_1^{-2}y_2^{-1}}{\sqrt{2\pi \Gamma_{12}}} \exp\left[ -\frac{\left\{\log(y_2/y_1) + \Gamma_{12}/2\right\}^2}{2\Gamma_{12}}  \right], \quad (y_1,y_2) \in \mathcal E,
\end{align}
and $\Lambda(1,1) = 2 \Phi\left(\sqrt{\Gamma_{12}}/2\right)$, where $\Phi$ is the standard normal distribution function.
The extension of H\"usler--Reiss distributions to random fields are
Brown--Resnick processes \citep{bro1977, kab2009}, which are widely used models for spatial
extremes.

\end{example}

\begin{example}[Bivariate Pareto distribution]\label{ex_biv}
  In the general bivariate case $d=2$, due to homogeneity, the density $\lambda$ of the exponent measure can be characterised by a univariate distribution. Indeed, for any positive random variable $U^1_2$ with density $f_{U^1_2}$ and $\mathbb E U^1_2 =1$,
  \begin{align}\label{dens_biv}
    \lambda(y_1,y_2) = y_1^{-3} f_{U^1_2}(y_2/y_1), \quad (y_1,y_2) \in \mathcal E,
  \end{align}
  satisfies conditions (L1) and (L2) above and thus defines a valid exponent measure density.
  We call $U^1_2$ the extremal function at coordinate $2$, relative to coordinate~$1$ \citep[cf.,][]{dom2013,dom2016}. Equivalently, we  can write the density in terms of
  the extremal function $U^2_1$ at coordinate~$1$, relative to coordinate~$2$, as
  $\lambda(y_1,y_2) = y_2^{-3} f_{U^2_1}(y_1/y_2)$, $(y_1,y_2) \in \mathcal E$, and
  $U^2_1$ is related to $U^1_2$ via the measure change $\mathbb P(U^2_1 \leq z) = \mathbb E( \einsfun\{1/U^1_2 \leq z\} U^1_2)$, $z > 0$.
\end{example}
  The above is a general construction principle, since every valid exponent measure density can be obtained in this way. The bivariate H\"usler--Reiss distribution in~\eqref{eq:HRbiv} corresponds to the case of log-normal $U^1_2$ and $U^2_1$, but many other parametric and non-parametric examples are available \citep[e.g.,][]{BoldiDavison, CooleyDavisNaveau,BallaniSchlather,car2014}.

\subsection{Graphical models}
\label{graph_models}

A graph $\mathcal G = (V, E)$ is defined as a set of nodes $V = \{1,\dots, d\}$, also called vertices, together with a set of edges $E \subset V\times V$ of pairs of distinct
nodes. The graph is called undirected if for two nodes $i,j\in V$, $(i,j)\in E$ if and only if $(j,i)\in E$. For notational convenience, for undirected graphs we sometimes represent edges as unordered pairs $\{i,j\}\in E$. When counting the number
of edges, we count $\{i,j\}\in E$ such that each edge is considered only once.
A subset $C\subset V$ of nodes is called complete
if it is fully connected in the sense that $(i,j)\in E$ for all $i,j\in C$. We denote by $\mathcal C$ the set of all cliques, that is, the complete subsets that are not properly contained within any other complete subset. 

To each node $i\in V$ we associate a random variable $X_i$ with continuous state space $\mathcal X_i \subset \mathbb R$. The resulting random vector $\g X = (X_i)_{i\in V}$ takes
values in the Cartesian product $\mathcal X = \times_{i\in V} \mathcal X_i$.
Suppose that $\g X$ has a positive and continuous Lebesgue density $f_{\g X}$ on $\mathcal X$. For three disjoint subsets $A, B, C \subset V$ whose union is $V$, we say that $\g X_A$
is conditionally independent of $\g X_C$ given $\g X_B$ if the density factorizes as
\begin{align}\label{cond_ind}
  f_{\g X}(\g x) = \frac{f_{{A\cup B}}(\g x_{A\cup B}) f_{{B\cup C}}(\g x_{B\cup C})}{f_{{B}}(\g x_B)},
\end{align}
and we write $\g X_A \ci \g X_C \mid \g X_B$.
If $B=\emptyset$, then~\eqref{cond_ind} amounts to independence of $\g X_A$ and $\g X_C$.

The random vector $\g X$ is said to be a probabilistic graphical model
on the graph $\mathcal G = (V, E)$ if its distribution satisfies
the pairwise Markov property relative to $\mathcal G$, that is,
$X_i\ci X_j\mid \g X_{\setminus \{i,j\}}$ for all $(i,j)\notin E$.
If in addition, for any disjoint subsets $A,B,C\subset V$ such that
$B$ separates $A$ from $C$ in $\mathcal G$, $\g X_A\ci \g X_C\mid \g X_B$, then
$\g X$ is said to obey the global Markov property relative to $\mathcal G$.
Since $f_{\g X}$ is positive and continuous, it follows from the Hammersley--Clifford theorem  \citep[cf.,][Theorem 3.9]{Lauritzen}
that the two Markov properties are equivalent, and they are further equivalent to the factorization of the density
\begin{align}\label{eq:HCthm}
  f_{\g X}(\g x)= \prod_{C\in \mathcal C}{\psi_{C}( \g x_C)},\qquad  \g x\in \mathcal X,
\end{align}
for suitable functions $\psi_C$ on $\times_{i \in C} \mathcal X_i$.
If the graph $\mathcal G$ is decomposable, then this factorization can be rewritten
in terms of marginal densities
 \begin{align}\label{eq:DecHCthm}
 f_{\g X}(\g x) = {\prod_{C\in\mathcal C} f_{C}(\g x_C)\over \prod_{D\in\mathcal D} f_{D}(\g x_D)},\qquad  \g x\in \mathcal X,
\end{align}
where $\mathcal D$ is a multiset containing intersections between the cliques called separator sets; see \cite{Lauritzen} and Appendix \ref{def_graphs} for the definition of decomposability and separator sets.

\begin{example}\label{gauss}
  We recall that for a normal distribution $\g W = (W_i)_{i\in V}$ with invertible covariance matrix $\Sigma$, the precision matrix $\Sigma^{-1}$ contains the conditional independencies, or equivalently the graph structure, since for $i,j\in V$,
  $$ W_i \ci W_j \mid \g W_{\setminus \{i,j\}} \quad \iff \quad   \Sigma^{-1}_{ij} = 0.$$
\end{example}

\section{Conditional independence for threshold exceedances}
\label{sec:CI}
The notion of conditional independence has not been exploited in extreme value theory. In fact, for max-stable distributions it only leads to trivial probabilistic structures \citep{papastathopoulos2016conditional}. An exception are directed acyclic graphs for max-linear models studied in \cite{gis2018} and \cite{gis2018a}, which do however not admit densities. 

We therefore approach the problem from the perspective of threshold exceedances. Since the notion of independence is only defined on product spaces, the meaning of conditional independence is not straightforward for a multivariate Pareto distribution $\g Y = (Y_i)_{i \in V}$, $V=\{1,\dots,d\}$, with support on the $L$-shaped space $\mathcal L = \{\g x \in \mathcal E: \|\g x\|_\infty > 1\}$.
In this section we show that there is nevertheless a natural definition of conditional independence for $\g Y$. 
To this end, we restrict $\g Y$ to product spaces. For any $k\in V$,
we consider the random vector $\g Y^k$ defined as $\g Y$ conditioned on the event that $\{Y_k > 1\}$. Clearly, $\g Y^k$ has support on the product space $\mathcal L^k = \{ \g x \in \mathcal L: x_k > 1\}$ (cf., Figure \ref{supports}) and it admits the density
\begin{align}\label{fk_def}
  f^k(\g y) = \frac{f_{\g Y}(\g y)}{\int_{\mathcal L^k} f_{\g Y}(\g y) \mathrm d \g y} = \lambda(\g y), \quad \g y \in \mathcal L^k,
\end{align} 
since $\int_{\mathcal L^k} f_{\g Y}(\g y) \mathrm d \g y = 1 / \Lambda(\g 1) $ because of (L2) in Section~\ref{mevd}.
From \eqref{fk_def} we see that the densities $f^1,\dots, f^d$ coincide with $\lambda$
on the intersection of their supports. Therefore there are no problems with lack of self-consistency
as for instance in the model of \cite{HeffernanTawn2004}.

For any set $I \subset V$ with $k\in I$, the marginal $\g Y_I^k$ has density
\begin{align*}
     f^k_I(\g y_I) =  \int_{[0,\infty)^{d-|I|}} \lambda(\g y) \mathrm d \g y_{\setminus I} = \lambda_I(\g y_I), \quad \g y_I\in \mathcal L_I^k,
\end{align*}
which is homogeneous of order $-(|I|+1)$ on $\mathcal L_I^k = \{ \g x_I \in \mathcal L_I: x_k > 1\}$; see~\eqref{lambda_margin}. This is however not the case if $k\notin I$, since integration over $\g y_{\setminus I}$ then includes $y_k$ whose domain is $(1,\infty)$ in $\mathcal L^k$, and thus in general
$f^k_I(\g y_I) \neq \lambda_I(\g y_I)$, $\g y_I\in [0,\infty)^{|I|}.$
\vspace*{-1em}
\begin{figure}[ht]
\centering
{\includegraphics[height=5cm]{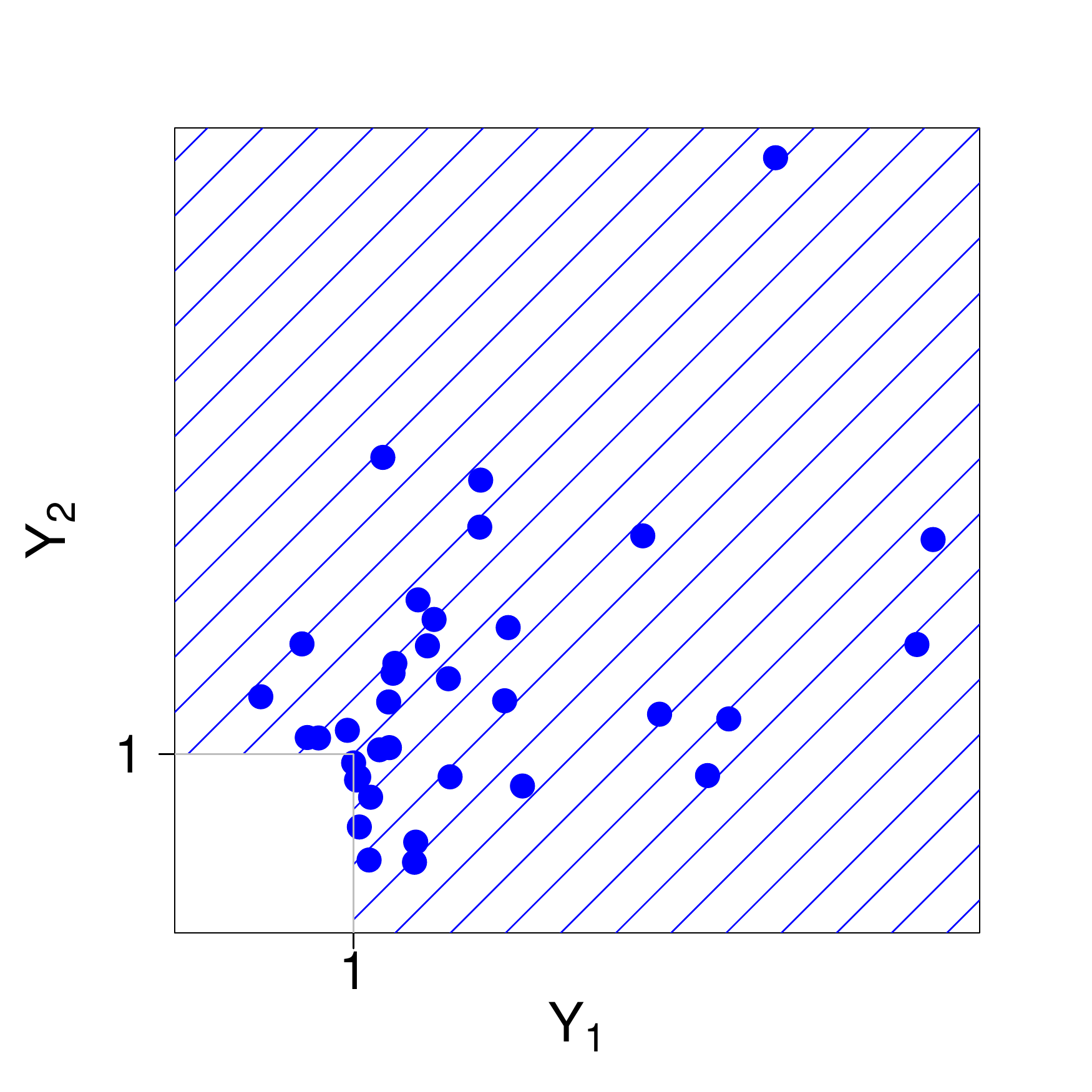}}%
{\includegraphics[height=5cm]{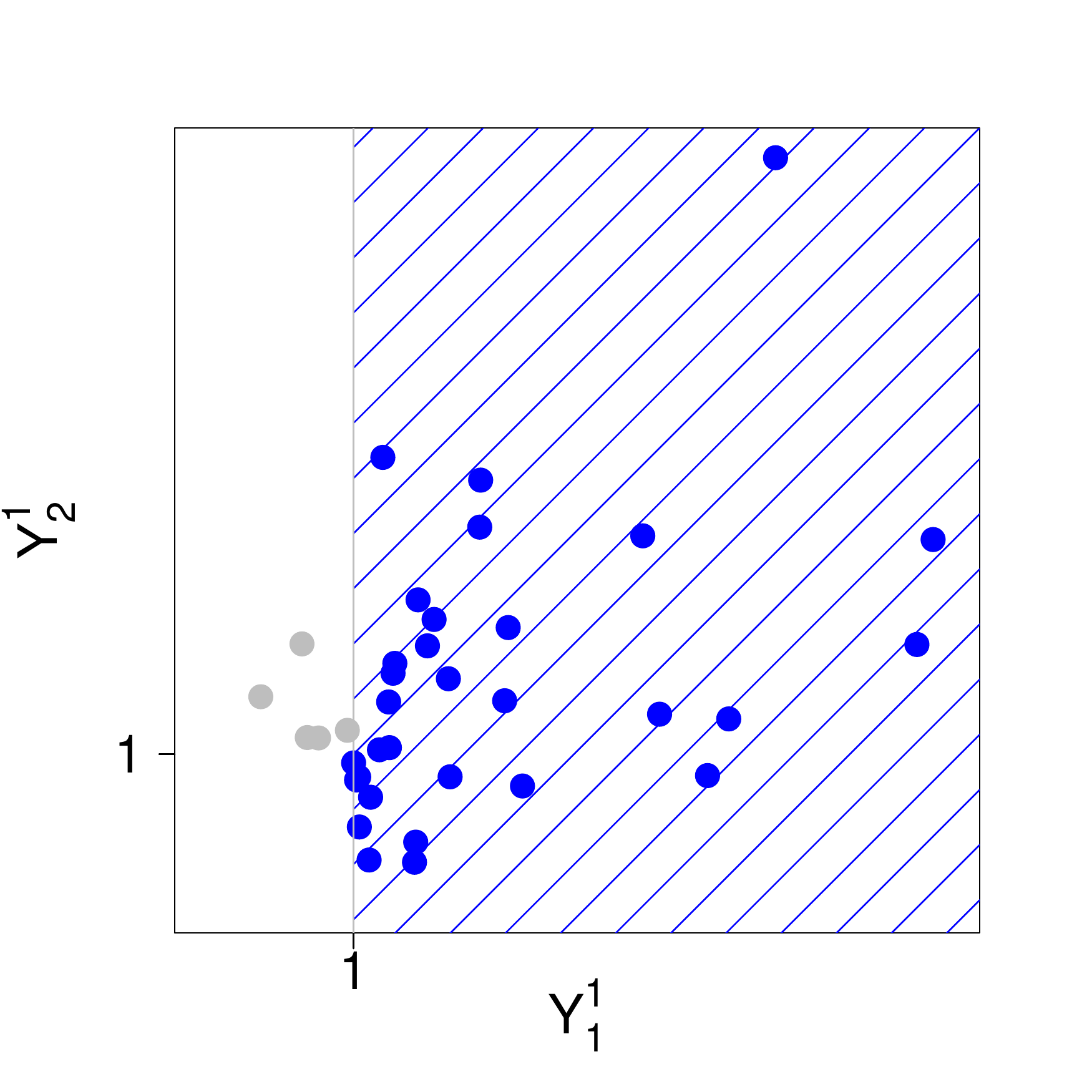}}%
{\includegraphics[height=5cm]{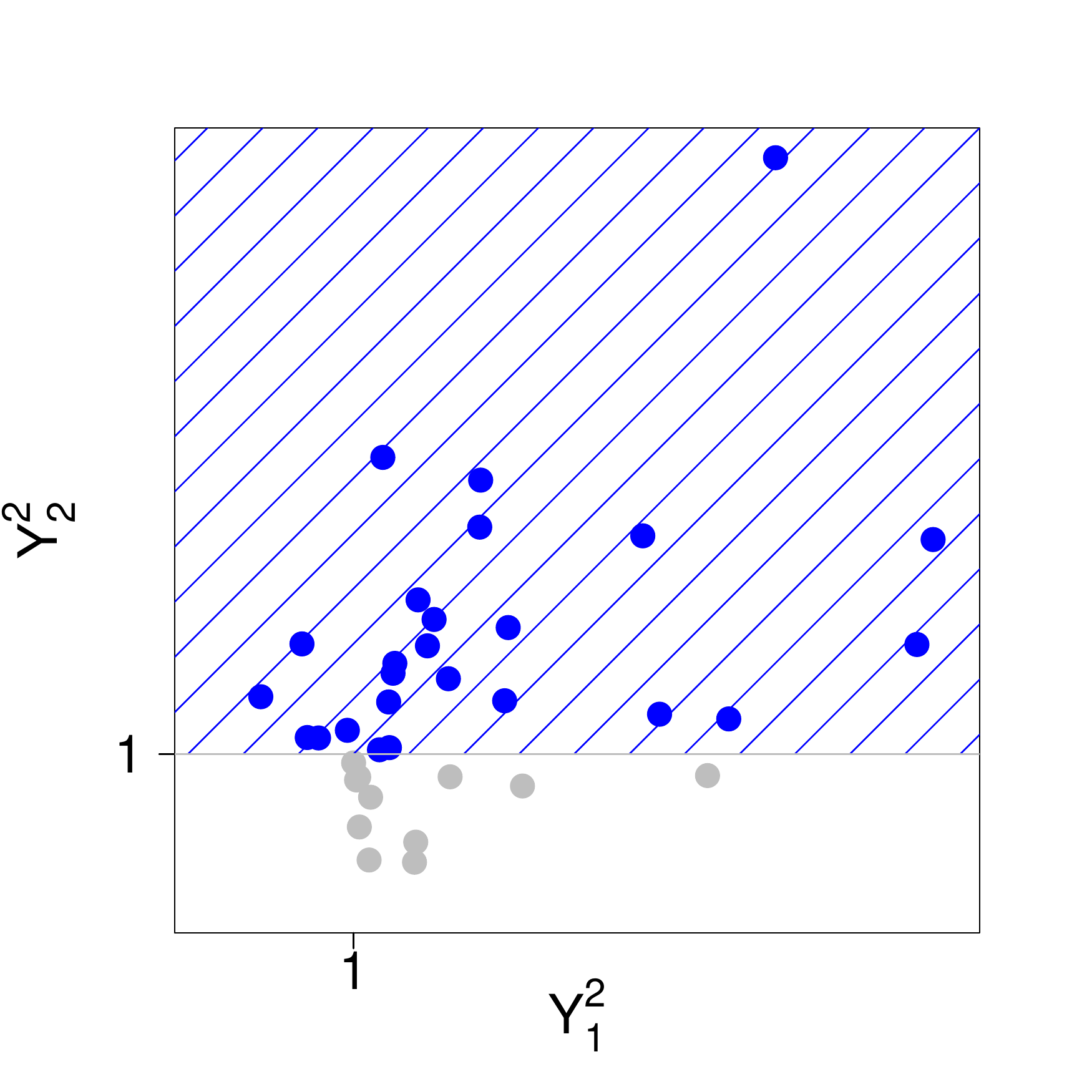}}%
\caption{Blue hatched areas are the support $\mathcal L$ of $\g Y$ (left), and the supports $\mathcal L^1$ of $\g Y^1$ (center) and $\mathcal L^2$ of $\g Y^2$ (right); points are samples of $\g Y$.}
\label{supports}
\end{figure} 
\vspace*{-.5em}
\begin{definition}\label{CI_def}
Suppose that $\g Y$ is multivariate Pareto and admits a positive and continuous density $f_{\g Y}$ on $\mathcal L$, and let $A,B,C \subset V$ be non-empty disjoint subsets whose union is $V=\{1,\ldots,d\}$.
We say that $\g Y_A$ is conditionally independent of $\g Y_C$ given $\g Y_B$ if
\begin{align}\label{eq:citail3}
  \forall k\in  \{1,\dots, d\}: \quad \g Y^k_A\ci \g Y^k_C \mid \g Y^k_B.
\end{align}
In this case we write $\g Y_A\perpe \g Y_C\mid \g Y_B$.
\end{definition}

In fact, this definition can be shown to be equivalent to a slightly weaker condition, and
to a factorization of the exponent measure density $\lambda$.
\begin{proposition}\label{prop_equiv}
Let $f_{\g Y}$ and the sets $A,B,C$ be as in the above definition, then $\g Y_A\perpe \g Y_C\mid \g Y_B$ is equivalent to any of the following two conditions.
\begin{enumerate}
\item[(i)]
  \begingroup\abovedisplayskip=-5pt 
  \begin{align}\label{ci_weak}    
    \exists k\in B : \quad    \g Y^k_A\ci \g Y^k_C \mid \g Y^k_B.
  \end{align}
  \endgroup
\item[(ii)]
  The density of the exponent measure factorizes as 
  \begin{align}\label{lambda_fac}
    \lambda(\g y) = {\lambda_{A\cup B}(\g y_{A\cup B}) \lambda_{B\cup C}(\g y_{B\cup C}) \over  \lambda_{B}(\g y_B)}, \qquad \g y \in \mathcal L.
  \end{align}
\end{enumerate}

\end{proposition}

A natural question is whether one can extend the definition of  $\g Y_A\perpe \g Y_C\mid \g Y_B$ to the case where $B = \emptyset$, meaning that $\g Y_A$ and $\g Y_C$ are independent on $\mathcal L$. In terms of the original definition, that would mean
that for any $k\in V$, $f^k(\g y) = f^k_A(\g y_A)f^k_C(\g y_C)$ for all $\g y \in \mathcal L^k$. Without losing generality, suppose $k\in A$, then
$f^k_C(\g y_C) = \lambda(\g y_A,\g y_C)/ \lambda_A(\g y_A)$ for any $\g y_A \in \mathcal L^k_A$ and $\g y_C \in [0,\infty)^{|C|}$. Therefore $f^k_C$ would be  homogeneous of order $-|C|$ and thus not integrable on $[0,\infty)^{|C|}$, a contradiction. In the next section we show that this property implies
that all graphical models defined in terms of the conditional independence $\perpe$ must be connected.

\section{Graphical models for threshold exceedances}
\label{graph_models_extremes}

The notion of conditional independence allows us to define graphical models
for threshold exceedances. 
As before, let $f_{\g Y}$ be the positive and continuous density on $\mathcal L$ of a multivariate Pareto distribution $\g Y$, proportional to the density $\lambda$ of the exponent measure $\Lambda$, and homogeneous of order $-(d+1)$.
Let $\mathcal G = (V, E)$ be an undirected graph with nodes $V = \{1,\dots, d\}$ and edge set $E$. Similarly to the classical probabilistic graphical models described in Section \ref{graph_models}, we say that $\g Y$ satisfies the pairwise Markov property on $\mathcal L$ relative to $\mathcal G$ if
\begin{align}\label{egm2}
  Y_i \perpe Y_j \mid \g Y_{\setminus \{i,j\}}, \quad  (i,j)\notin E,
\end{align}
that is, $Y_i$ and $Y_j$ are conditionally independent in the sense of Definition \ref{CI_def} given all other nodes whenever there is no edge between $i$ and $j$ in $\mathcal G$.
By definition, this is equivalent to saying that~$\g Y^k$ satisfies the usual pairwise Markov property on $\mathcal L^k$ relative to $\mathcal G$ for all $k\in V$.
The global Markov property for $\g Y$ is defined similarly.

\begin{definition}
  Let $\mathcal G=(V,E)$ be an undirected graph. If the multivariate Pareto distribution $\g Y$ with
  positive and continuous density $f_{\g Y}$ satisfies the pairwise
  Markov property~\eqref{egm2} relative to $\mathcal G$, we call the distribution of $\g Y$ an extremal graphical model with respect to $\mathcal G$.  
\end{definition}
For a decomposable graph $\mathcal G$ we obtain a factorization of the density $f_{\g Y}$ similar to
the classical Hammersley--Clifford theorem, showing that the Definition \ref{CI_def} of
conditional independence is natural for multivariate Pareto distributions. Let $\mathcal C$ and $\mathcal D$ be the sequences of cliques and separators of~$\mathcal G$, respectively, satisfying the running intersection property \eqref{rip} in Appendix~\ref{def_graphs}.

\begin{theorem}\label{thm:equivFactorHomog}
Let $\mathcal G=(V,E)$ be a decomposable graph and suppose that $\g Y$
has a multivariate Pareto distribution with positive and continuous
density $f_{\g Y}$ on $\mathcal L$. Denote the corresponding exponent measure and its density by $\Lambda$ and $\lambda$, respectively. Then the following are equivalent. 
\begin{enumerate}
\item[(i)] \label{eq:thmFactHomi} The distribution of ${\g Y}$ satisfies the pairwise Markov property relative to $\mathcal G.$
  \item[(ii)] \label{eq:thmFactHomii} The distribution of ${\g Y}$ satisfies the global Markov property relative to $\mathcal G.$
\item[(iii)] \label{eq:thmFactHomiii} The density $f_{\g Y}$ factorizes according to $\mathcal G$, that is,
\begin{align}\label{eq:producthc}
f_{\g Y}(\g y)= \frac{1}{\Lambda(\mathbf 1)} {\prod_{C\in \mathcal C} \lambda_C(\g y_C) \over \prod_{D\in \mathcal D} \lambda_D(\g y_D)},\quad \g y\in \mathcal L,
\end{align}
where the marginals $\lambda_I$ are positive, continuous and homogeneous of order $-(|I|+1)$ for any $I\subset V$.
\end{enumerate}
In all cases, the graph $\mathcal G$ is necessarily connected.
\end{theorem}

\begin{remark}
  The above theorem shows that only connected extremal graphical models can arise. This is related to
  the assumption of multivariate regular variation in \eqref{mrv} that implies asymptotic dependence
  between all components. Loosely speaking, unconnected components would correspond
  to asymptotically independent components.
\end{remark}

\begin{remark}
  If the graph $\mathcal G$ in the above theorem is non-decomposable, it is expected that the density $f_{\g Y}$ still factorizes
  into factors on the cliques of the graph. These factors can however no longer be identified
  with marginal densities, and it is an open problem whether they can be chosen to be homogeneous.
\end{remark}

Since $\mathcal L$ is not a product space, unlike in the classical Hammersley--Clifford theorem for
decomposable graphs in~\eqref{eq:DecHCthm}, the factors in the factorization of the density $f_{\g Y}$ in~\eqref{eq:producthc} are not the marginals $f_{I}$ but the marginals of the exponent measure density $\lambda_I$. It holds however that $f_{I}(\g y_I) = \lambda_I(\g y_I) / \Lambda_I(\g 1)$ for all $\g y_I \in \mathcal L_I \subset \{\g x_I: \g x \in \mathcal L\}$.

As a first application, the above theorem allows us to formally analyse the
conditional independencies and graphical structures of models in the multivariate extreme value literature.

\begin{example}\label{time_series}
  One of the simplest examples of a graph is a chain, that is, 
  $$E = \{\{1,2\},\{2,3\},  \dots, \{d-1,d\}\}.$$
  \cite{coles1991modelling} proposed a model that factorizes with respect to this chain where all bivariate marginals are logistic (cf., Example \ref{ex_log}). This was extended to general bivariate marginals in \cite{smi1997}.
  More generally, in the study of extremes of stationary Markov chains the limiting objects are so-called tail chains. The latter induce multivariate Pareto distributions that can readily be seen to factorize with respect to a chain; see \cite{smi1992} \cite{bas2009} and \cite{jan2014}.
\end{example}

\begin{example}
  It turns out that many of the multivariate models in the literature do not have any
  conditional independencies, that is, their underlying graph is fully connected.
  For instance, this holds for the logistic multivariate Pareto distribution in Example \ref{ex_log}, the Dirichlet mixture model in \cite{BoldiDavison},
  and the pairwise beta distribution in \cite{CooleyDavisNaveau}.
\end{example}

\begin{example}\label{HR_stable}
Similar to Gaussian distributions, an appealing property of the H\"usler--Reiss model is
its stability under taking marginals.
Indeed, for any $I \subset V$ and $k \in I$  
the marginal density of the exponent measure is
\begin{align*}
  \lambda_I(\g y_I) &= \int_{[0,\infty)^{d-|I|}} \lambda_I(\g y) \mathrm{d}\g y_{\setminus I}= y_k^{-2}\prod_{i\in I\setminus\{k\}} y_i^{-1} \phi_{|I|-1}\left\{\g{\tilde y}_{I\setminus\{k\}}; \Sigma^{(k)}_{I}\right\},
\end{align*}
with the notation of Example \ref{ex:HR}, where $\Sigma^{(k)}_{I}$ is the matrix in~\eqref{sigma_k} induced by the submatrix~$\Gamma_{I}$. Thus, $f_{I}( \g y_I) = \lambda_I(\g y_I) / \Lambda_I(\mathbf 1)$ is the density of the $|I|$-dimensional H\"usler--Reiss Pareto distribution with parameter matrix $\Gamma_{I}$.

By Theorem \ref{thm:equivFactorHomog}, the density of a H\"usler--Reiss distribution that satisfies the pairwise Markov property relative to some decomposable graph $\mathcal G$
factorizes into lower-dimensional H\"usler--Reiss distributions. The explicit formula is given in Corollary \ref{cor_HR_dens} in Appendix~\ref{app_decomp}.
\end{example}

Theorem \ref{thm:equivFactorHomog} can also be seen as a construction principle to
build new classes of extreme value distributions in a modular way by combining
lower-dimensional marginals. 
The following corollary shows how a  multivariate Pareto distributions can be defined
that factorizes according to a desired underlying graphical structure.
This is particularly useful in high-dimensional problems to ensure model sparsity.

\begin{corollary}\label{cor_construction}
  Let $\mathcal G$ be a decomposable and connected graph and suppose that
  $\{\lambda_I : I\in \mathcal C \cup \mathcal D\}$ is a set
    of valid, positive and continuous exponent measure densities in the sense of (L1) and (L2) in Section \ref{mevd}.
    For $ D\subset C$, $D\in \mathcal D,$ $C\in\mathcal C$, assume that they satisfy the consistency constraint
    \begin{align}\label{eq_consistency}
      \lambda_D(\g y_D)=\int_{[0,\infty)^{|C\setminus D|}} \lambda_C(\g y_C) \,\mathrm d\g y_{C\setminus D}.
    \end{align}
    The density of a valid $d$-dimensional exponent measure $\Lambda$ is then given by
    \begin{align*}
      \lambda(\g y)= {\prod_{C\in \mathcal C} \lambda_C(\g y_C) \over \prod_{D\in \mathcal D} \lambda_D(\g y_D)},\quad \g y\in \mathcal L,
    \end{align*}
    and the function $f_{\g Y}(\g y) = \lambda(\g y) / \Lambda(\mathbf 1)$, $\g y \in\mathcal L$, is the density of a multivariate Pareto distribution satisfying the pairwise Markov property relative to $\mathcal G$.
\end{corollary}

\subsection{Tree graphical models}
\label{trees}

A tree is a special case of a decomposable graphical model that is connected and has no cycles. All cliques are then of size two and the separators contain only one node. Let $\mathcal T = (V,E)$ be an undirected tree with nodes $V=\{1,\dots, d\}$ and edge set $E \subset V\times V$.
Suppose that $\g Y = (Y_i)_{i\in V}$ follows a multivariate Pareto distribution on $\mathcal L$ with density $f_{\g Y}$ that is an extremal graphical model with respect to the tree $\mathcal T$. Theorem \ref{thm:equivFactorHomog} yields the factorization
\begin{align}\label{tree_fact}
f_{\g Y}(\g y) = \frac{1}{\Lambda(\g 1)} \prod_{\{i,j\}\in E} {\lambda_{ij}(y_i,y_j) \over y_i^{-2} y_j^{-2}} \prod_{i\in V} y_i^{-2},
\end{align}
where $\lambda_{ij} = \lambda_{\{i,j\}}$ are the bivariate marginals of the exponent measure density $\lambda$ corresponding to $\g Y$.   
The formula~\eqref{tree_fact} allows the extension of the modelling approach by \cite{smi1997} described in Example \ref{time_series} from time series to general tree structures. Such tree models are able to represent more complex dependencies and, moreover, are suitable beyond temporal data for multivariate or spatial applications.

Thanks to the relatively simple structure of a tree, more explicit results
can be derived than for general graphical models. To this end, we define a new, directed tree $\mathcal T^k = (V,E^k)$ rooted at an arbitrary but fixed node $k\in V$. The edge set $E^k$ consist of all edges $e\in E$ of the tree $\mathcal T$
pointing away from node $k$; see Figure \ref{dir_tree} for an example with $k=2$.
For the resulting directed tree we define a set $(U_e)_{e\in E^k}$ of independent random variables, where for $e = (i,j)$, the distribution of $U_e = U^i_j$ is the extremal function of $\lambda_{ij}$ at coordinate $j$, relative to coordinate $i$; see~\eqref{dens_biv} in Example \ref{ex_biv} for the definition of extremal functions.
The following stochastic representation of the random vectors $\g Y^k$, $k\in V$, provides a better understanding of the stochastic structure of multivariate Pareto distributions factorizing on a tree, and it is the main ingredient for efficient simulation in Section \ref{simulation}.

\begin{figure}[ht]
\centering
{\includegraphics[clip,height=5cm, page=6]{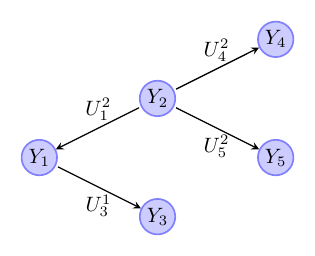}}%
\caption{A tree $\mathcal T^2$ rooted at node $k=2$ with the extremal functions in Proposition \ref{prop_tree} on the edges.}
\label{dir_tree}
\end{figure}

\begin{proposition}\label{prop_tree}
  Let $\g Y$ be a multivariate Pareto distribution with positive and continuous density on $\mathcal L$  that factorizes with respect to the tree $\mathcal T$. With the notation above, and for a standard Pareto distribution $P$, we have the joint stochastic representation for $\g Y^k$ on $\mathcal L^k$
\begin{align}\label{tree_rep}
  Y^k_i \, \stackrel{d}{=} 
  \begin{cases}
     P, & \text{for } i = k,\\
     P \times \prod_{e\in \ph(ki)} U_e, &\text{for } i\in V\setminus \{k\},
  \end{cases}
\end{align}
where $\ph(ki)$ denotes the set of edges on the unique path from node $k$ to node $i$ on the tree~$\mathcal T^k$.
\end{proposition}

\begin{remark}
  The same object as in \eqref{tree_rep} has been obtained in \cite{seg2019} as the limit of regularly varying
  random vectors that satisfy a Markov condition on a tree. In analogy to the tail chains
  in Example \ref{time_series}, they term it a tail tree.  
\end{remark}

\begin{example}\label{ex_tree}
  Suppose all bivariate marginals $\lambda_{ij}$ for $\{i,j\}\in E$ of a tree Pareto model are of logistic type with parameter $\theta_{ij}\in(0,1)$ as defined in Example \ref{ex_log}. This tree logistic model
  is a generalization of the chain logistic model considered in \cite{coles1991modelling}. In this symmetric case, the extremal functions $U^i_j$ and $U^j_i$ have the same distribution with stochastic
  representation $F/G$, where $F$ follows a Fr\'echet$(1/\theta, c_\theta)$ distribution with
  scale parameter $c_\theta = \Gamma(1-\theta)^{-1}$ and $(G/c_\theta)^{-1/\theta}$ follows a Gamma$(1 - \theta,1)$ distribution, where we abbreviated $\theta = \theta_{ij}$ and $\Gamma$ is the Gamma function.
\end{example}
  Similarly we can define a H\"usler--Reiss tree model, or use asymmetric models for $\lambda_{ij}$ such as the Dirichlet model in \cite{BoldiDavison} for some or all of the edges $\{i,j\}\in E$. In asymmetric models, the extremal functions $U^i_j$ and $U^j_i$ have in general different distributions. We refer to Section 4 in \cite{dom2016} for explicit formulas for extremal function distributions of commonly used model classes.

\subsection{H\"usler--Reiss graphical models}
\label{HRgraph}

In many ways, the class of H\"usler--Reiss distributions introduced in Example \ref{ex:HR} can be seen as the natural analog of Gaussian distributions in the world of asymptotically dependent extremes. They are parameterized by the variogram of Gaussian distributions, and their
statistical inference \citep{wad2013, Engelke2015} and exact simulation \citep{dom2016} involves tools that are closely related to
the corresponding methods for normal models.

Despite the similarities to Gaussian distributions, there are subtle but important differences that render analysis and statistical inference of  H\"usler--Reiss distributions more difficult.
In order to characterise conditional independence and graphical structures in these models, we first recall some results related to the original construction. The max-stable H\"usler--Reiss distribution has a stochastic representation as componentwise maxima 
\begin{align}\label{eq:HRmax}
  \g Z = \max_{l\in\mathbb N}  U_l \exp\left\{ \g W_l - \diag(\Sigma) / 2 \right\}, 
\end{align}
where $\{U_l:l\in\mathbb N\}$ is a Poisson point process on $[0,\infty)$ with intensity $u^{-2}\mathrm d u$, and $\g W_l$ are independent copies of a $d$-dimensional normal distribution $\g W$ with zero mean and covariance matrix $\Sigma$. Subtracting 
$\mathbb E (\exp\g W) = \diag(\Sigma) / 2$ in the exponent normalises the marginals of  $\g Z$ to be standard Fr\'echet.
 \cite{kab2009} show that the distribution of $\g Z$ only depends on the strictly conditionally negative definite variogram matrix of~$\g W$,
$$\Gamma_{ij} = \mathbb E (W_i - W_j)^2, \qquad i,j\in V.$$
Importantly, this implies that the representation in~\eqref{eq:HRmax} is not unique since any centred, possibly degenerate normal distribution $\g W$ with variogram matrix $\Gamma$ leads to the same max-stable H\"usler--Reiss distribution. Let 
\begin{align}\label{gamma_set}
  \mathcal S_\Gamma = \{\Sigma\in\mathbb R^{d\times d} \text{ covariance matrix}: \einsfun \diag(\Sigma)^\top +  \diag(\Sigma)\einsfun^\top - 2 \Sigma = \Gamma\},
\end{align}
be the set of all covariance matrices that correspond to the same variogram matrix $\Gamma$; see Appendix \ref{link_vario}.
The H\"usler--Reiss Pareto distribution $\g Y$ associated with $\g Z$ is defined by its density in Example \ref{ex:HR}, which is also parameterized by $\Gamma$.
We recall that for a normal distribution~$\g W$ with invertible covariance matrix $\Sigma$, the precision matrix $\Sigma^{-1}$ contains the conditional independencies; see Example \ref{gauss}.
A first, naive guess would be that the graph structure of~$\g W$ used in the construction
of $\g Z$ directly translates
into the extremal graph structure of the H\"usler--Reiss Pareto distribution $\g Y$. This is however not the case.
\begin{example}\label{ex_sig}
  We consider three examples for $\g W$ in the representation~\eqref{eq:HRmax} with $d=4$.
  \begin{enumerate}
    \item
      Let $W_i$, $i=1,\dots, 4$, be independent standard normal distributions, then
      $\Sigma^{-1} = \diag(1,\dots,1)$ and $\Gamma_{ij} = 2$ if $i\neq j$ and zero otherwise.
      The graph underlying the distribution of $\g W$ is the graph with four unconnected nodes. The graph of the corresponding H\"usler--Reiss Pareto distribution $\g Y$ turns out to be the fully connected graph on the left-hand side of Figure \ref{graphs_ex}. 
    \item
     Consider the centred normal distribution $\g W$ with precision matrix and variogram matrix
     \[
    \Sigma^{-1} = 
    \begin{bmatrix}
      12 & -4 & -4 & -1 \\
      -4 & 2 & 1 & 0 \\
      -4 & 1 & 2 & 0 \\
      -1 & 0 & 0 & 1 
    \end{bmatrix},
    \qquad 
    \Gamma = 
    \begin{bmatrix}
      0 & 1 & 1 & 1 \\
      1 & 0 & 2 & 2 \\
      1 & 2 & 0 & 2 \\
      1 & 2 & 2 & 0 
    \end{bmatrix},
    \]
    respectively.
    The Gaussian graphical model is the graph in the centre of Figure \ref{graphs_ex} with an additional edge between the nodes $2$ and $3$. On the contrary, the corresponding H\"usler--Reiss model factorizes according to the graph in the centre of Figure \ref{graphs_ex}.
    \item
      Consider the centred normal distribution $\g W$ with precision matrix and variogram matrix
      \[
      \Sigma^{-1} = 
    \begin{bmatrix}
      2 & -0.5 & -0.5 & 0 \\
      -0.5 & 1 & 0 & -0.5\\
      -0.5 & 0 & 1 & -0.5 \\
      0 & -0.5 & -0.5 & 1 
    \end{bmatrix},
    \qquad 
    \Gamma = 
    \begin{bmatrix}
      0 & 1.5 & 1.5 & 2 \\
      1.5 & 0 & 2 & 1.5 \\
      1.5 & 2 & 0 & 1.5 \\
      2 & 1.5 & 1.5 & 0 
    \end{bmatrix},
    \]
    respectively.
    It can be checked that both the Gaussian and the  H\"usler--Reiss graphical model are as in the right-hand side of Figure \ref{graphs_ex}. Also note that this graph is not decomposable.
  \end{enumerate}
\end{example}

\begin{figure}[ht]
\centering
{\includegraphics[clip,height=3cm, page=1]{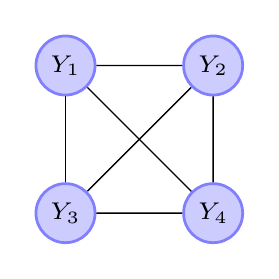}}%
{\includegraphics[clip,height=3cm, page=2]{HR_graphs.pdf}}%
{\includegraphics[clip,height=3cm, page=3]{HR_graphs.pdf}}%
\caption{The H\"usler--Reiss graphical models described in Example \ref{ex_sig}.}
\label{graphs_ex}
\end{figure}

The above examples show that it is not possible to simply transfer the Gaussian graphical model of the covariance matrix $\Sigma$ in the construction~\eqref{eq:HRmax}
to the extremal graphical structure of the corresponding H\"usler--Reiss Pareto distribution. This is not surprising since the covariance matrices in the set $\mathcal S_\Gamma$ can have very different graph structures, but all lead to the same H\"usler--Reiss graphical model.
There is however a set of particular matrices that allow us to identify conditional independencies and thus the graphical structure of a H\"usler--Reiss Pareto distribution.
Recall the definition of $\Sigma^{(k)}\in\mathbb R^{(d-1)\times (d-1)}$ in~\eqref{sigma_k}. The same matrix including the $k$th row and column 
\begin{align}\label{full_sigma}
  \tilde \Sigma^{(k)}  =\frac{1}{2} \{\Gamma_{ik}+\Gamma_{jk}-\Gamma_{ij}\}_{i,j\in V} \in\mathbb R^{d\times d},
\end{align}
is degenerate since the $k$th component has zero variance, but it is a valid choice in the construction~\eqref{eq:HRmax}, that is, $\tilde \Sigma^{(k)} \in \mathcal S_\Gamma$, for any $k\in V$. Let $\g W^k$ be a centred normal distribution with covariance matrix $\tilde \Sigma^{(k)}$ and note that $W^k_k = 0$ almost surely.
For a random variable $P$ with standard Pareto distribution, independent of $\g W^k$, it can be seen that 
\begin{align}\label{Yk_rep}  \g Y^k  \stackrel{d}{=} P \exp\left\{ \g W^k - \Gamma_{\cdot\, k} / 2 \right\},
\end{align}
by comparing the density of the right-hand side with~\eqref{eq:fYHR} and noting that $\diag(\tilde \Sigma^{(k)}) = \Gamma_{\cdot\, k}$. This together with the original definition of conditional independence in~\eqref{eq:citail3} suggests that the matrices $\Sigma^{(k)}$ contain the graphical structure of 
a H\"usler--Reiss Pareto distribution.

We denote the precision matrix of $\Sigma^{(k)}$ by $\Theta^{(k)} = (\Sigma^{(k)})^{-1}$. For notational convenience, the indices of the matrices $\Sigma^{(k)}$ and $\Theta^{(k)}$ range in $\{1,\dots, d\}\setminus \{k\}$ instead of $\{1,\dots, d-1\}$.

\begin{lemma}\label{cor_HR}
  For $k,k'\in V$, $k \neq k'$, the precision matrices $\Theta^{(k)}$ and $\Theta^{(k')}$
  satisfy for $i,j \in V \setminus \{k'\}$ 
  \begin{align*}
    \Theta^{(k')}_{ij}&= \Theta^{(k)}_{ij},   &\text{ if }  i,j \neq  k,\\
    \Theta^{(k')}_{ik}&=  -\sum_{l\neq k} \Theta^{(k)}_{il} , &\text{ if } i\neq k, j = k,\\
    \Theta^{(k')}_{kk} &=  \sum_{l,m \neq k} \Theta^{(k)}_{lm} , & \text{ if } i,j = k.
  \end{align*}  
\end{lemma}

The above lemma is of independent interest since it explains the link between the precision
matrices $\Theta^{(k)}$ for different different $k\in V$. The proof uses blockwise inversion
of the precision matrices.
This result is the crucial ingredient to characterise conditional independence in H\"usler--Reiss models.
\begin{proposition}\label{prop:CI_HR}
  For a H\"usler--Reiss Pareto distribution $\g Y$ with parameter matrix $\Gamma$, it holds for $i,j\in V$ with $i\neq j$, and for any $k\in V$, that
  \begin{align}\label{eq:CI_HR}
    Y_i\perpe Y_j\mid \g Y_{\setminus \{i,j\}} \quad \iff \quad
    \begin{cases}
       \Theta^{(k)}_{ij}= 0,  &\text{ if } i,j \neq k,\\
       \sum_{l\neq k} \Theta^{(k)}_{lj} = 0, & \text{ if } i = k, j\neq k,\\
       \sum_{l\neq k} \Theta^{(k)}_{il} = 0, & \text{ if } j = k, i \neq k. 
    \end{cases} 
  \end{align}
\end{proposition}

For any $k \in V$, the single matrix $\Theta^{(k)}$ contains all information on conditional independence of~$\g Y$. Conditional independence concerning the $k$th component
is encoded in the row and column sums of $\Theta^{(k)}$, and it might sometimes be easier to switch
to another representation $\Theta^{(k')}$, $k'\neq k$, where it simply figures as a zero entry.
In Example \ref{ex_sig} we can now easily determine the graphical model $\mathcal G = (V,E)$ for each of the three H\"usler--Reiss Pareto distributions. For a given $\Sigma$ we first compute the matrix $\Gamma$ as in~\eqref{gamma_set}, then transform it by~\eqref{sigma_k} to obtain $\Sigma^{(k)}$ for any $k\in V$ and use Proposition~\ref{prop:CI_HR} to decide whether $(i,j) \in E$ for all $i,j\in V$.
These transformations are implemented in our R-package \texttt{graphicalExtremes} \citep{graphicalExtremes}.

\begin{example}
  In spatial extreme value statistics, the finite dimensional distributions of the Brown--Resnick process \citep{kab2009} at locations $t_1,\dots, t_d\in \mathbb R^D$ are H\"usler--Reiss distributed with
  variogram matrix $\Gamma_{ij} = \gamma(t_i - t_j)$, $i,j\in\{1,\dots, d\}$, where $\gamma$ is a
  variogram function on $\mathbb R^D$. The most commonly used model is the fractal variogram
  family $\gamma_\alpha(h) = \| h\|_2^\alpha$, for some $\alpha\in (0,2]$. The corresponding
  $d$-variate H\"usler--Reiss distribution does not have conditional independencies and
  its graph is thus fully connected. The only exception is the case of the original
  Brown--Resnick process in \cite{bro1977} with $\alpha=1$ and $D=1$, where the corresponding
  graph is a chain as in Example \ref{time_series}.  
\end{example}

In this section, we have so far not required that the underlying graph $\mathcal G$ is decomposable. If this is the
case then, as shown in Example \ref{HR_stable}, Theorem \ref{thm:equivFactorHomog} 
implies that the density of the H\"usler--Reiss graphical model factorizes into lower-dimensional
H\"usler--Reiss densities; see Corollary \ref{cor_HR_dens} in Appendix~\ref{app_decomp}.

\section{Statistical inference for block graphs}
\label{sec:stat}

\subsection{Model construction}
\label{construction}

The notion of conditional independence and graphical models for multivariate Pareto distributions
allows the construction of new statistical models with two major advantages. First, sparsity can be imposed
on the model, which is a crucial ingredient for tractable and parsimonious models in higher dimensions.
Second, under certain graphical structures, the model parameters can be estimated separately on lower-dimensional subsets of the data.

We consider here, and throughout the rest of the paper, decomposable and connected graphs $\mathcal G=(V,E)$ with clique set $\mathcal C$ and separator set $\mathcal D$, where all separators in $\mathcal D$ are single nodes.
Such graph structures with singleton separator sets are known as block graphs \citep[cf.,][]{har1963} and have already been seen to have appealing properties for discrete
data \citep{loh2013}. In our case, they are a convenient way of restricting the model complexity in order to obtain a tractable class of extremal graphical models. In fact, Corollary \ref{cor_construction} provides a simple construction principle for multivariate Pareto distributions that factorize with respect to the block graph $\mathcal G$.
\begin{itemize}
  \item[i)]
    For each clique $C\in \mathcal C$, choose possibly different parametric families of valid exponent measure densities $\{\lambda_C(\cdot; \theta_C) : \theta_C \in \Omega_C\}$ for suitable parameter spaces $\Omega_C$. If $\mathcal G$ is a tree $\mathcal T$, then this reduces to choosing $d-1$ bivariate
    exponent measure densities $\lambda_{ij}$, for each $\{i,j\} \in E$; see Example \ref{ex_biv} for a general representation of such densities.
  \item[ii)]
    Since all separator sets consist of a single node, the consistency
constraint~\eqref{eq_consistency} is trivially fulfilled 
as a consequence of (L1) and (L2) in Section \ref{mevd} and the fact that $\lambda_D(y_D) = y_D^{-2}$ for all $D\in \mathcal D$. 
  \item[iii)]
    For any fixed combination of parameters $\theta = (\theta_C)_{C\in \mathcal C} \in \Omega = \times_{C\in \mathcal C} \Omega_C$, the product of the normalised lower-dimensional exponent measure densities, 
\begin{align}\label{graph_dens}
  f_{\g Y}(\g y; \theta)= \frac{1}{\Lambda(\g 1; \theta)}\prod_{C\in \mathcal C}{ \lambda_C(\g y_C; \theta_C) \over \prod_{j \in C} y_j^{-2}} \prod_{i\in V} y_i^{-2} ,\quad \g y\in \mathcal L,
\end{align}
defines a valid $d$-variate Pareto distribution factorizing according to the graph $\mathcal G$, which is a member of the parametric family parameterized by $\theta \in \Omega$. For a tree $\mathcal T$, this reduces to the density in~\eqref{tree_fact}.
\end{itemize}

Concrete examples for this construction are tree logistic or tree H\"usler--Reiss models as described in
Example~\ref{ex_tree}, where all cliques have the same type of distributions. The above
construction is much more flexible, as it allows us to use different distribution families for the
different cliques. Moreover, some, or even all of the cliques may be modeled by non-parametric
methods; see \cite{laf2012} for non-parametric tree models in the non-extreme case. In this direction, there is a line of research on kernel-based estimation of exponent measure densities \citep[cf.,][]{car2014, mar2017, kir2018a} that could be used as clique models. We will not follow this approach here.

In the graphical models above, the dependence inside each clique is modeled directly, whereas dependence between
components from different cliques is implicitly implied by the conditional independence structure of
the graph. Even if all cliques are modeled with the same type of parametric family, the joint distribution~\eqref{graph_dens} is typically not of this distribution type. For a tree logistic distribution, for instance, this can easily be seen by comparing its density~\eqref{tree_fact} with that of $d$-variate logistic distribution in Example \ref{ex_log}. The latter only has one parameter governing the whole $d$-dimensional dependence structure, whereas the tree has $d-1$ logistic parameters $\{\theta_{ij};\{i,j\}\in E\}$ and thus much higher flexibility. 

An important exception is the family of H\"usler--Reiss distributions, which is stable under taking marginal distributions; see Example \ref{HR_stable}. The following proposition shows that for a given graphical structure as above, if all cliques have H\"usler--Reiss distributions, then so has the full $d$-dimensional model. This is the converse of Corollary \ref{cor_HR_dens}
in Appendix \ref{app_decomp}.         

\begin{proposition}\label{mat_compl}
  Let $\mathcal G = (V,E)$ be a block graph as above, and suppose that on each clique $C\in \mathcal C$,
  $\g Y$ has a $|C|$-variate H\"usler--Reiss distribution with exponent measure
  density $\lambda_C(\cdot; \Gamma^{(C)})$ parameterized by a $|C|\times |C|$-dimensional variogram matrix $\Gamma^{(C)}$. Then there exists a unique solution to the problem:
  \begin{align}
    \notag & \text{find a } d\times d\text{-dimensional variogram matrix } \Gamma,\\
    \label{mat_cons}&\text{subject to } \begin{cases}
      \Gamma_{ij} = \Gamma^{(C)}_{ij} , & \text{ for } i,j\in C \text{ and all } C\in\mathcal C,\\
      \Theta^{(k)}_{ij} = 0, & \text{ for all }k\in V, i,j\neq k \text{ and } (i,j)\notin E,
    \end{cases}
  \end{align}
  with the notation from Proposition \ref{prop:CI_HR}. 
  The corresponding $d$-variate H\"usler--Reiss distribution factorizes according
  to the graph $\mathcal G$ into the lower-dimensional H\"usler--Reiss densities on the cliques.
\end{proposition}
This is a matrix completion problem for variograms similar to what \cite{dem1972} introduced
for covariance matrices. In our case, the graph is decomposable and the above result relates to
the marginal problem studied in \cite{kel1964} and \cite{daw1993}. For H\"usler--Reiss marginals
on block graphs we even see that the implied $d$-dimensional distribution is again H\"usler--Reiss.
We give a direct, constructive proof in Appendix \ref{appendix_proofs}.
  This provides a method to construct high-dimensional H\"usler--Reiss distributions
  out of many low-dimensional ones.
  The full $d$-variate H\"usler--Reiss model without any conditional independencies has $d(d-1)/2$ parameters. A H\"usler--Reiss distribution as in Proposition \ref{mat_compl} that factorizes on a block graph with click set $\mathcal C$ has only
  $$ \dfrac{1}{2}\sum_{C \in \mathcal C} |C|(|C| -1)$$
  parameters, which can be much smaller than $d(d-1)/2$.

\subsection{Estimation}
\label{estimation}

Extremal graphical models can be used to build parsimonious statistical models for the tail of
a multivariate random vector. In this section we discuss how the model parameters can be estimated efficiently by considering each clique distribution separately.

Let $\g X = (X_j)_{j\in V}$, $V=\{1,\dots, d\}$, be a random vector in the max-domain
of attraction of the max-stable random vector $\g Z$ as in~\eqref{mrv}, with marginal distribution $X_j$ in the max-domain of attraction of a generalized extreme value distribution with shape parameter $\xi_j$, $j\in V$. Equivalently, there exist a sequence of high thresholds $\g t_u = (t_{u1}, \dots, t_{ud})$ with $t_{uj}$ tending to the upper endpoint of $X_j$ as $u\to \infty$, and positive normalizing functions $\sigma_u = (\sigma_{u1}, \dots, \sigma_{ud})$, such that the distribution of exceedances converges weakly
\begin{align}
   \frac{\g X - \g t_u}{\sigma_u}  \mid \| \g X / \g t_u \|_\infty > 1  \longrightarrow  \frac{\g Y^\xi-1}{\xi}, \quad u \to \infty,
\end{align}
where $\g Y$ is the multivariate Pareto distribution associated with $\g Z$. We assume $\g Y$
to be in the model class of the previous section with density~\eqref{graph_dens}, and for now
we suppose that the underlying graph $\mathcal G = (V,E)$ is known and fixed. The conditional density 
of $\g X - \g t_u$ given that $\| \g X / \g t_u\|_\infty > 1$ is then approximated by
\begin{align}
  f_{\g Y}\left\{ \left(1+ \xi \frac{\g x}{\sigma_u} \right)^{1/\xi}; \theta \right\} \prod_{j\in V} \frac{1}{\sigma_{uj}} \left(1+ \xi_j \frac{x_j}{\sigma_{uj}} \right)^{1/\xi_j-1}.
\end{align}
This density can be used to estimate jointly the marginal parameters $(\sigma_{uj}, \xi_j)$, $j\in V$, and the dependence parameter vector $\theta = (\theta_C)_{C\in \mathcal C}$ of $f_{\g Y}$.

In the sequel we concentrate on estimation of the dependence, and we therefore assume that
the marginal parameters are known or have been estimated separately. As described in Section~\ref{mevd}, we can then normalise $\g X$ to standard Pareto marginals, in which case $\xi_j=1$, $t_{uj} = u$ and $\sigma_{uj} = u$ for all $j\in V$. We recover the standardized setting of~\eqref{mpd} considered throughout the paper, where $\g X/ u$ given that $\| \g X \|_\infty > u$ converges  to $\g Y$, whose likelihood is proportional as a function of $\theta$ to
\begin{align}\label{llh_contr}
  & f_{\g Y}(\g y; \theta) \propto  \frac{1}{Z_\theta} \prod_{C\in \mathcal C} \frac{\lambda_{C}(\g y_C; \theta_C)}{\Lambda_C(\g 1; \theta_C)}, \quad  Z_\theta =  \frac{\Lambda(\g 1; \theta)}{\prod_{C\in \mathcal C} \Lambda_{C}(\g 1; \theta_C)}.
\end{align}
Direct maximization of the likelihood with contributions~\eqref{llh_contr} for each data point
is tedious since the normalizing constant $Z_\theta$ contains all parameters and does not factorize.
Fortunately the class of block graphs has the property that we can estimate the parameters $\theta_C$ of each
$\lambda_C$ separately, without having to enforce the consistency constraints at the separator sets.
In fact, we use the following observation. If $\g X$ is in the domain of attraction of the
family of multivariate Pareto distributions $\{f_{\g Y}(\cdot; \theta) : \theta \in \Omega\}$,
then for a fixed clique $C\in \mathcal C$, the subvector $\g X_C$ is in the domain
of attraction of $\{f_{C}(\cdot; \theta_C) : \theta_C \in \Omega_C\}$, and the distribution of the normalised exceedance $\g X_C / u \mid \|\g X_C\|_\infty > u$ is approximated for large $u$ by $\g Y_C$ with density
\begin{align}\label{clique_contr}
  f_{C}(\g y_C; \theta_C)= \frac{\lambda_C(\g y_C; \theta_C) }{\Lambda_C(\g 1; \theta_C)}  ,\quad \g y_C\in \mathcal L_C;
\end{align}
see~\eqref{margin_pareto} in Section \ref{mevd}.
We can therefore obtain an estimate of $\theta_C$ based only on data of the components in $C$,
whose dimension is typically much smaller than the dimension $d$ of the full graph.
Estimating the cliques separately might in principle result in a loss of
estimation efficiency compared to using the joint likelihood~\eqref{llh_contr}.
The normalizing constant $Z_\theta$ does however not contain much information
on the parameter $\theta$ and the maximum likelihood estimate
using $f_{\g Y}(\g y; \theta)$ is generally very close to the estimate obtained
by maximizing separate likelihoods based on~\eqref{clique_contr}. We discuss this
point in the simulation study in Section~\ref{simulation_study}.

In practice, some components of $\g X$ might not have converged to the limiting distribution~$\g Y$.
In order to avoid biased estimates of the dependence parameters $\theta_C$, it has become
a standard approach to apply censoring to the data; see  \cite{led1997}, \cite{smi1997}. For a data point $\g X_C$ with $\|\g X_C\|_\infty > u$ for a high threshold $u>0$, define $J$ to be the set of indices $j\in C$ such that $Y_j < 1$, i.e., $X_j < u$. For this data point we use the censored likelihood contribution
\begin{align}\label{cens_contr}
  f_{C}^{\cens}(\g y_C; \theta_C) = \int_{[0,1]^{|J|}} f_{C}(\g y_C; \theta_C) \mathrm d \g y_J ,\quad \g y_C\in \mathcal L_C,
\end{align}
which uses for all $j\in J$ only the information that this component of $\g Y_C$ is smaller than $1$, but not its exact value.
For explicit forms of the censored likelihoods for many parametric models see \cite{dom2016a} and \cite{kir2018}.

For $n$ independent data $\g y^{(h)}\in \mathcal L$, $h=1,\dots, n$, of $\g X / u \mid \|\g X\|_\infty > u$,
for each clique $C$ we define $\widehat \theta_C$ as the maximizer of the censored log-likelihood
\begin{align}\label{clique_llh}
  L(\theta_C; \g y^{(1)}, \dots, \g y^{(n)}) = \sum_{\g y^{(h)}\in \mathcal L_C} \log\{ f_{C}^{\cens}(\g y_C^{(h)}; \theta_C) \},
\end{align}
where $\mathcal L_C =\{\g y\in \mathcal L: \exists j\in C \text{ s.t. } y_j > 1 \}$,
and each $\g y^{(h)}_C$ has its own censoring set $J^{(h)}\subset C$.

Maximum likelihood estimation is only one possibility to infer the parameters $\theta_C$ based on exceedances of $\g X_C$ and the limiting distribution \eqref{clique_contr}. Alternative methods use $M$-estimators \citep{ein2012, ein2016} or proper scoring rules \citep{def2018}.

\subsection{Model selection}
\label{model_selection}

Up to now we have assumed that a graphical structure $\mathcal G$ was \textit{a priori} given and we analysed models that factorize with respect to this structure. 
In many applications the underlying graph structure is unknown and should be
learned in a data-driven way.
Theorem \ref{thm:equivFactorHomog} implies that all extremal graphical structures are connected,
and a simple and flexible class of connected graphs are trees; see Section \ref{trees}.
It is thus natural to first build a suitable tree as a baseline model, and then extend
the tree by adding additional edges in order to obtain more complex graphs.

Since trees are a special case of general graphical models, there are specific methods to learn these simpler structures. The notion of a minimum spanning tree is crucial \citep{kruskal1956shortest}. Let $\mathcal G_0 = (V, E_0)$ be
the fully connected graph on $V=\{1,\dots, d\}$ with edge set $E_0 = \{(i,j): i,j\in V\}$. Suppose that a positive weight $w_{ij}>0$ is attached to each edge $(i,j)\in E_0$ of $\mathcal G_0$. This number can be seen as the length of the edge $(i,j)$ or the distance between nodes $i$ and $j$, and it is assumed that $w_{ij} = w_{ji}$ and $w_{ii}=0$, $i,j\in V$. The minimum spanning tree is the
tree $\mathcal T_{\MST} = (V,E_{\MST})$ with $E_{\MST}\subset E_0$, that minimizes the sum of weights on that tree, i.e.,
\begin{align}\label{Tmin}
  \mathcal T_{\MST} = \argmin_{\mathcal T = (V,E)} \sum_{(i,j)\in E} w_{ij}. 
\end{align}
If all edges of $\mathcal G_0$ have distinct lengths, then $\mathcal T_{\MST}$ is unique.
This minimization problem can be solved efficiently by the greedy algorithms proposed
in \cite{kruskal1956shortest} or \cite{pri1957}.

The weights $w_{ij}$ determine the tree structure and should be chosen carefully.
A common approach in graphical modelling is to search the conditional independence structure that maximizes the likelihood, \citep[cf.,][Chapter 11]{cow2006}. Such a tree is also called a Chow--Liu tree \citep{cho1968}.
We fix a parametric family of bivariate Pareto distributions that is used for all
pairs of nodes $\{f(\cdot; \theta_{ij}) : \theta_{ij} \in \Omega\}$. For $n$ independent data $\g y^{(h)}$, $h=1,\dots, n$, the maximal log-likelihood
of a fixed tree within this parametric class is essentially the sum over
the maximized clique log-likelihoods in~\eqref{clique_llh} over all edges of this tree. 
In order to find the tree that maximizes the log-likelihood over all trees and all distributions
in this parametric family, we therefore find the minimum spanning tree in~\eqref{Tmin}
with weights
\begin{align}\label{Tweights}
  w_{ij} = - L(\widehat \theta_{ij}; \g y^{(1)}, \dots, \g y^{(n)}) - 2\sum_{y_i^{(h)} > 1} \log y_i^{(h)} - 2\sum_{y_j^{(h)} > 1} \log y_j^{(h)} ,
\end{align}
where we include the censored marginal densities $y_i^{-2}$ and $y_j^{-2}$ in~\eqref{graph_dens}
for the clique $\{i,j\}$, since now the edges are no longer fixed but parameters of the optimization.
The resulting tree ${\mathcal{T}}_{\MST}$ is the baseline model for the data.
If the model fit is not satisfactory, it is possible to extend this tree to graphs
with more complex structures by adding additional edges.
The family of H\"usler--Reiss distributions is particularly appealing since the bivariate marginals
remain in the same class.
We illustrate this model extension through a greedy forward selection
in Section~\ref{simulation_study}.

The different multivariate Pareto models can then be compared by the Akaike information criterion \citep{kir2018}, 
\begin{align}\label{AIC}
  \AIC = 2p - 2 L(\widehat \theta; \g y^{(1)}, \dots, \g y^{(n)}),
\end{align}
where $p$ is the number of parameters in the respective model, and the second term
is twice the negative log-likelihood based on the censored version of~\eqref{llh_contr}, evaluated at the optimized parameters of each clique. 

\subsection{Exact simulation}\label{simulation}

Exact simulation of a max-stable random vector $\g Z$ relies on the notion of
extremal functions \citep{dom2013a}. The extremal function of $\g Z$, or of its associated multivariate Pareto distribution $\g Y$, relative to coordinate $k\in V$ is the $d$-dimensional random vector $\g U^k$ with $U^k_k =1$ such that the exponent measure density
of $\g Z$ can be written as
\begin{align}\label{def_ef}
  \lambda(\g y) = y_k^{-(d+1)} f_{\g U^k_{\setminus k}}(\g y_{\setminus k} / y_k).
\end{align} 
The distributions of the extremal functions $\g U^k$, $k\in V$, for most commonly used models have explicit forms and are derived in Section 4 of \cite{dom2016}. Theorem 2 in the same paper relates the distribution of the so-called spectral measure to these extremal functions. Together
with the following representation of $\g Y$, this enables
simulation of multivariate Pareto distributions by rejection sampling. Recall that for any $k\in V$, the random vector $\g Y^k$ is defined as $\g Y$ conditioned on the event that $\{Y_k > 1\}$.
\begin{lemma}\label{lem_extr_func}
  The distribution of the extremal function $\g U^k$ of $\g Y$ relative to coordinate $k\in V$ is given by the distribution of $\g Y^k / Y^k_k$. Independently, let $P$ be a standard Pareto random variable and $T$ uniformly distributed on $\{1,\dots, d\}$. We then have for any Borel set $A \subset \mathcal L$
  \begin{align}\label{Y_rep}
\mathbb P(\g Y \in A) = \mathbb P\left( \frac{P\g Y^T}{\| \g Y^T\|_1} \in A  \,\bigl\vert\,   \frac{ P \|\g Y^T\|_\infty}{\| \g Y^T\|_1} > 1  \right).
  \end{align}
\end{lemma}
The above representation yields a simple algorithm for exact simulation of $\g Y$; see also \cite{def2018}.
\begin{algorithm}[Exact simulation of a multivariate Pareto distribution $\g Y$]\label{algo1}

1. Simulate a standard Pareto random variable $P$.\\
2. Simulate $T$ uniformly on $\{1,\dots, d\}$ and sample a realization of
the extremal function~$\g U^{T}$ relative to coordinate $T$.\\
3. If $\max\{P  \|\g U^{T}\|_\infty / \|\g U^{T}\|_1\} > 1$,\\
\phantom{3.} \quad return $ \g Y = P  \g U^{T} / \|\g U^{T}\|_1$ as realization of the multivariate Pareto distribution.\\
4. Else,\\
\phantom{4.} \quad reject the simulation and go back to step 1.
\end{algorithm}
The complexity of this simulation algorithm as a function of the dimension $d$ of the vector $\g Y$ is driven by the number of times one has to sample from one of the extremal functions $\g U^1, \dots, \g U^d$, since simulation of the variables $P$ and $T$ requires much less computational effort. Let $C_{\g Y}(d)$ denote the
number of extremal functions that have to be simulated in the above algorithm. The random variable $C_{\g Y}(d)$ follows a geometric distribution and from~\eqref{rej_prop} in the proof of Lemma~\ref{lem_extr_func} its expectation is
$$\mathbb E \{C_{\g Y}(d) \} = d / \Lambda(\g 1) \in [1,d].$$
The expected complexity
therefore depends on both the dimension and the strength of extremal dependence in $\g Y$. Weak dependence implies a large coefficient $\Lambda(\g 1)$ closer to $d$ and therefore reduces the computational effort required for exact simulation.
The simulation of multivariate Pareto distributions is in general computationally easier than
for the associated max-stable distribution $\g Z$. Indeed, exact simulation of the latter is also based on samples from a mixture of the $\g U^1, \dots, \g U^d$, and the fastest algorithm in
\cite{dom2016} has expected complexity $\mathbb E \{C_{\g Z}(d) \} = d$; see also \cite{die2015} and \cite{oes2018} for other exact simulation methods.

The complexity measures $C_{\g Y}(d)$ and $C_{\g Z}(d)$ only consider the number of extremal functions required for one exact simulation of $\g Y$ and $\g Z$, respectively. 
The computational effort of sampling $\g U^k$ can however be significantly lower
if $\g Y$ has a sparse structure. If $\g Y$ factorizes according to a graph, then, by the Definition \ref{CI_def} of conditional independence, the $\g Y^1, \dots, \g Y^d$ inherit the sparsity of this graph structure. This is particularly important in the case of trees and for H\"usler--Reiss distributions, as shown in the examples below.
It is important to note that more efficient simulation of the extremal functions speeds up exact simulation of the multivariate Pareto distribution $\g Y$, but also of the max-stable distribution~$\g Z$.

\begin{example}
  Suppose that $\g Y$ factorizes according to a tree $\mathcal T = (V, E)$.
  It follows from Proposition \ref{prop_tree} and
  Lemma \ref{lem_extr_func} that the extremal function $\g U^k$ relative to coordinate $k\in V$ is 
  $$\g Y^k / Y_k^k \stackrel{d}{=} \left(\prod_{e\in \ph(ki)} U_e \right)_{i\in V}.$$
  For exact simulation of $\g Y$ it therefore suffices to simulate the univariate 
  random variables $U_e$. This is feasible even in very large dimensions.  
\end{example}

\begin{example}
  If $\g Y$ has a H\"usler--Reiss distribution that factorizes on the 
  graph $\mathcal G=(V,E)$, then it follows from~\eqref{Yk_rep} that
  the extremal function $\g U^k$ relative to coordinate $k\in V$ is
  \begin{align*} 
    \g Y^k /Y^k_k \stackrel{d}{=} \exp\left\{ \g W^k - \Gamma_{\cdot\, k} / 2 \right\},
  \end{align*}
  where $\g W^k$ is a centred normal distribution with covariance matrix $\tilde \Sigma^{(k)}$ in~\eqref{full_sigma}; see also Proposition 4 in \cite{dom2016}. The normal distribution
  $\g W^k_{\setminus k}$ factorizes in the classical sense on the subgraph $\mathcal G_{\setminus k}$, and efficient simulation algorithms exist if the graph is sparse \citep[e.g.,][]{rue2005}. 
\end{example}

The exact simulation algorithms for both multivariate Pareto and max-stable distributions are implemented in our R-package \texttt{graphicalExtremes} \citep{graphicalExtremes}.

\subsection{Simulation study}\label{simulation_study}

We assess the efficiency of parameter estimation and model selection in the framework of
graphical models for extremes described in the previous sections. We fix a dimension $d$ of
variables or nodes $V=\{1,\dots, d\}$ and a block graph $\mathcal G = (V,E)$ as in Section~\ref{construction}. 
In this study we simulate samples directly from the limiting distribution $\g Y$ using the exact Algorithm~\ref{algo1}, but we use the censored estimation since this is
common practice in applications. 

We first choose $d=5$ and let $\mathcal G$ be the undirected version of the tree in Figure \ref{dir_tree}.
We simulate $n\in\{100,200\}$ samples $\g y^{(1)}, \dots, \g y^{(n)}$ of a H\"usler--Reiss distribution with parameter  matrix $\Gamma$ that factorizes according to $\mathcal G$.
The entries of $\Gamma$ need to be specified only on the submatrices  $\Gamma^{(C)}$ for all cliques $C \in \mathcal C$ of $\mathcal G$, since the solution to the matrix completion problem in Proposition~\ref{mat_compl} then yields the unique variogram matrix $\Gamma$. In this simulation we set
\begin{align}\label{gamma_tree}
  \Gamma = 
  \begin{bmatrix}
    0 & \mathbf{1} & \mathbf{2} & 2 & 3 \\
    1 & 0 & 3 & \mathbf{1} & \mathbf{2} \\
    2 & 3 & 0 & 4 & 5 \\
    2 & 1 & 4 & 0 & 3\\
    3 & 2 & 5 & 3 & 0
  \end{bmatrix},
\end{align}
where we only specified the four parameters $\Gamma_{ij}$ for $(i,j) \in E$, $i<j$, to the values in bold, and the rest of the matrix is implied by the graph structure.

In this dimension we can still maximize the censored version of the
joint likelihood~\eqref{llh_contr} to obtain an estimate $\widehat \Gamma_{ij}^{\text{joint}}$,
$\{i,j\} \in E$, of the parameters corresponding to the four edges of the tree.
We also obtain estimates $\widehat \Gamma_{ij}$,
$\{i,j\} \in E$, of the parameters of each clique separately
by maximizing the censored clique likelihood~\eqref{clique_llh}.
In both cases, the four estimated parameters yield estimates $\widehat \Gamma^{\text{joint}}$
and $\widehat \Gamma$ of the whole variogram matrix $\Gamma$ through the graph structure.
We repeat the simulation and estimation $200$ times and compare
the efficiency of both approaches in Figure \ref{fig_box}, displaying only the four free parameters that have actually been estimated.

The difference in efficiency between the joint and clique likelihoods seems to be
small or even negligible. This is due to two reasons. For non-censored points
the two likelihoods only differ by the normalizing constant $Z_\theta$. Since this
constant only measures the global strength of dependence and does not depend on
the data, it seems not very sensitive to changes in the parameter $\theta$. 
The second difference between the two approaches is that they use slightly different data. Consider a clique $C\in \mathcal C$ and the corresponding model
parameter~$\theta_C$. The joint likelihood uses all data $\g Y$ in the space $\mathcal L = \{\g y\in \mathcal E: \exists j\in V \text{ s.t. } y_j > 1\}$, but censors all components with $y_j\leq 1$. On the other hand, the clique likelihood uses the marginals $\g Y_C$
of all data $\g Y$ in $\mathcal L_C =\{\g y\in \mathcal L: \exists j\in C \text{ s.t. } y_j > 1 \}$. Consequently, the additional data used in the joint likelihood is
in $\mathcal L \setminus \mathcal L_C = \{\g y \in\mathcal  L: y_j \leq  1 \text{ for all } j\in C\}$. But the contribution to the joint likelihood of data in this set with regard
to the parameter $\theta_C$ is completely censored and does therefore not
add significant additional information. 
These two considerations underline that estimating the parameters for each clique
separately does not result in significant efficiency losses. This is one of the main
advantages of graphical models, namely that the distribution is defined locally
by the cliques and extends globally by the conditional independence structure.
In terms of computational aspects, the joint likelihood becomes infeasible even in
moderate dimensions, whereas the clique likelihood is applicable in
high dimensions as long as the cliques have small enough sizes. Moreover, the computations
for different cliques can be easily parallelized.
\begin{figure}[ht]
\centering
{\includegraphics[clip,height=.45\textwidth]{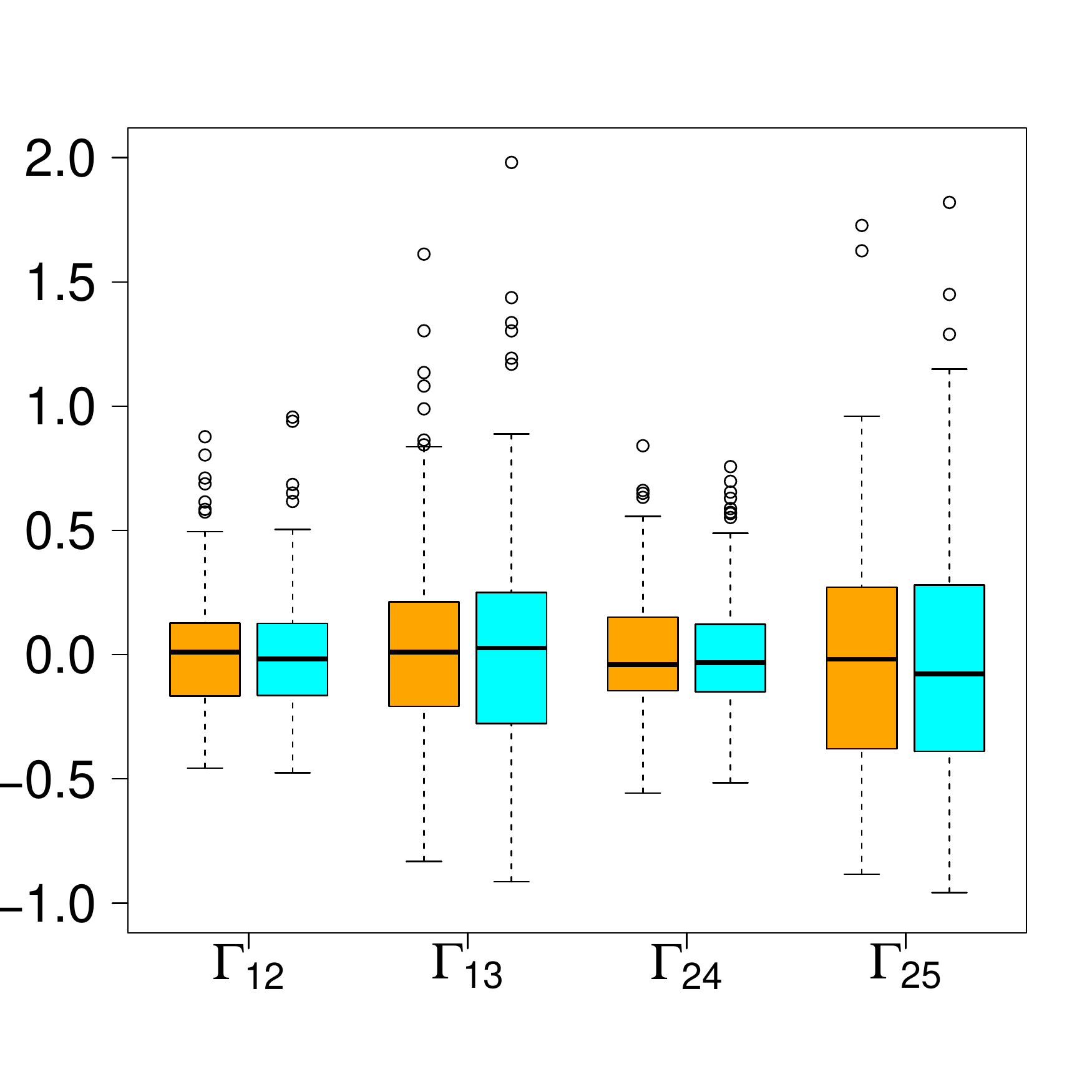}}%
{\includegraphics[clip,height=.45\textwidth]{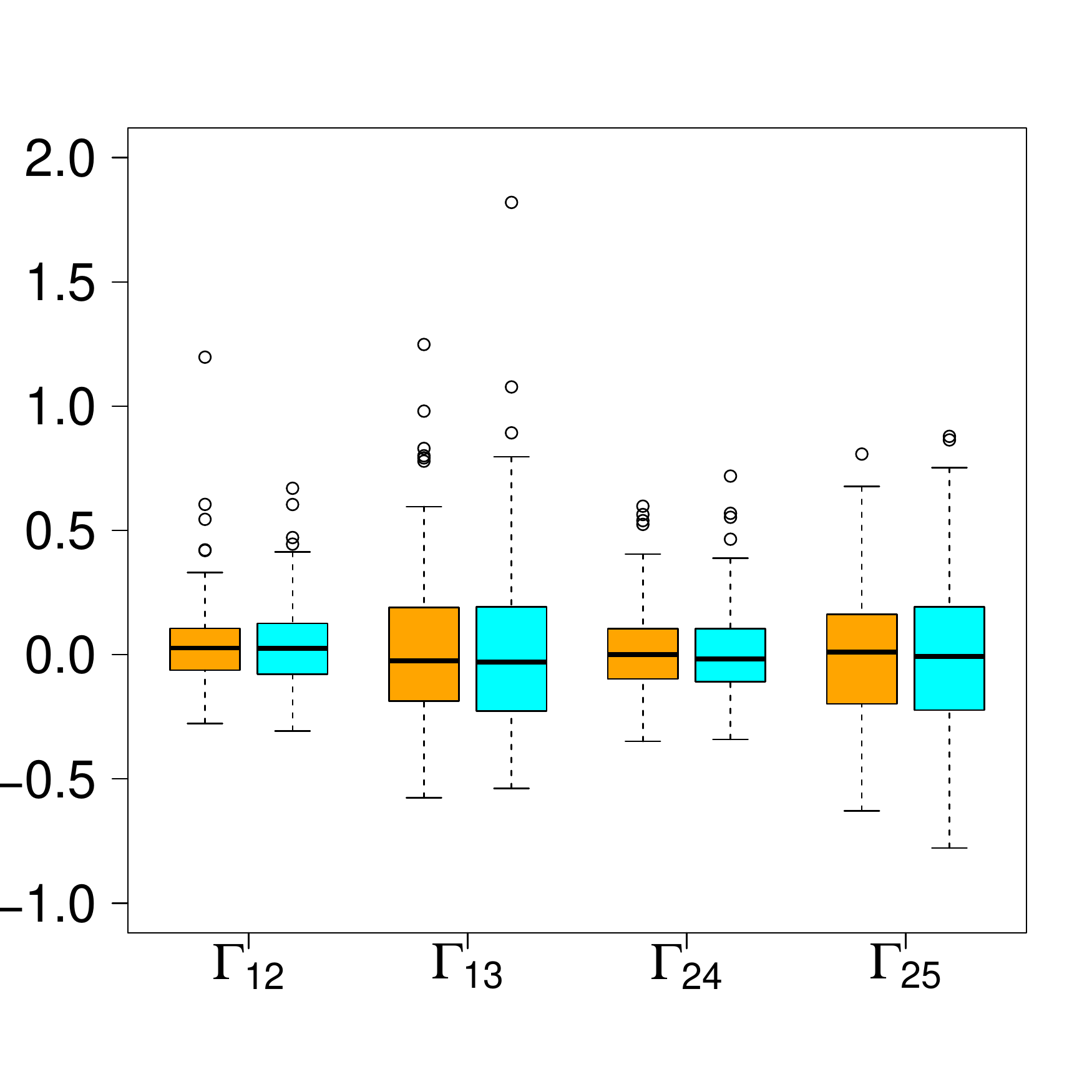}}%
\caption{Boxplots of errors of the four parameter estimates of the H\"usler--Reiss tree model in~\eqref{gamma_tree} based on joint (orange) and clique likelihood (blue) with sample size $n=100$ (left) and $n=200$ (right).}
\label{fig_box}
\end{figure}

For the second experiment we take $d=16$ and let $\mathcal G$ be the graph on the left-hand side of Figure \ref{fig_simu}, which is not a tree. We simulate $n=100$ samples of a H\"usler--Reiss distribution with parameter  matrix $\Gamma$ that factorizes according to $\mathcal G$. The parameters of
the $p=18$ edges are independently sampled from a uniform distribution on $(0.5,1)$, under the constraint that $\Gamma$ is conditionally negative definite on cliques with three nodes.  
We illustrate how we can choose the best graphical model, where
we restrict to block graphs as in Section~\ref{construction} with cliques of sizes two and three.
We first construct the minimum spanning tree as described in Section~\ref{model_selection}
within the class of H\"usler--Reiss distributions. The estimated edge set of this tree
is denoted by $E_1$. The~$15$ parameter estimates $\widehat \Gamma_{ij}$, $\{i,j\}\in E_1$ 
obtained by fitting the clique likelihoods of each clique of the tree yield
a unique estimate $\widehat \Gamma$ of the $d\times d$-dimensional variogram matrix; see Proposition \ref{mat_compl}. This tree model does not contain all edges of the true underlying graph. 
We therefore perform a greedy forward selection in order to add additional edges and improve the model. In each step, we define an enlarged edge set $E_{m+1} = E_m \cup \{i,j\}$, $m=1,2,\dots$, restricting to those edges $\{i,j\}$, $i,j\in V$, that still yield a block graph with cliques of maximal size three. We continue this process until no more edge can be added in this way. For the same parameter matrix $\Gamma$, we repeat the simulation and model selection 100 times. The right-hand side of Figure \ref{fig_simu} shows the graph with the selected edges,
where the line width of each edge indicates the number of times it has been selected among
the first $18$ edges. It can be seen that the graph structure is generally very well identified.
For each model and each repetition we also compute the resulting $\AIC$
according to~\eqref{AIC}. The proportion of times that the model with $\{15,\dots,20\}$ edges has the
smallest $\AIC$ are $\{0.01, 0.11, 0.23, 0.39, 0.23, 0.03\}$. Even though
the $\AIC$ is a criterion built for model estimation and not for identification \citep[cf.,][]{arl2010}, it seems to be well suited to select the correct degree of sparsity for this extremal graphical model.

\begin{figure}[ht]
\centering
{\includegraphics[trim=0cm 0 0cm 0cm, clip,height=5cm, page=10]{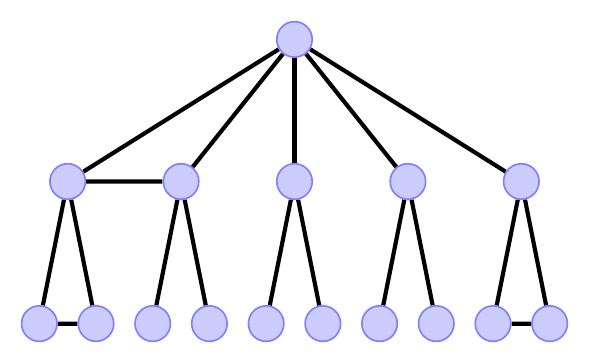}}%
{\includegraphics[trim=0cm 0 0cm 0cm, clip,height=5cm, page=4]{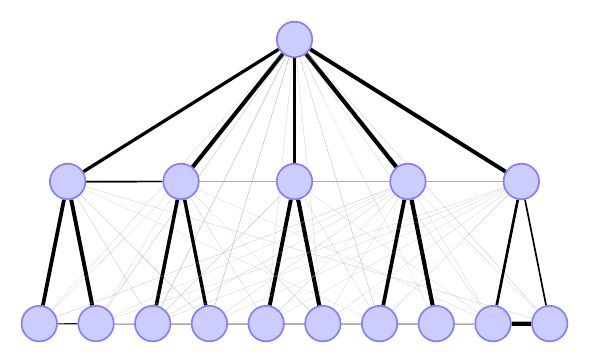}}%
\caption{True underlying graphical structure (left) and the estimated structure (right) in the second experiment, where the line width indicates the number of times the edge has been selected.}
\label{fig_simu}
\end{figure}

\section{Application}
\label{sec:appl}
We illustrate the applicability of extremal graphical models at the example of
river discharges in the upper Danube basin, a region that is prone to serious flooding. The data are provided by the Bavarian Environmental Agency (\texttt{http://www.gkd.bayern.de}) and we use $d=31$ gauging stations with $50$ years of common daily data from 1960--2009. The tree induced by the physical flow-connections at these stations is shown on the left-hand side of
Figure~\ref{fig_flow}, where the path $10\to 9 \to \dots \to 1$ is on the Danube and
the other branches are tributaries. The spatial extremal dependence structure of this data set has been studied
in \cite{asadi2015extremes} and we follow their preprocessing steps to make the results comparable. Out of all daily data only the three months June, July and August are considered since the most
severe floods occur in this period and are caused by heavy summer rain \citep{boe2006}. The $50\times 92 = 4600$ observations in these months are declustered in time in order to remove temporal dependence and to match slightly shifted peak flows at different locations. We refer to \cite{asadi2015extremes} for more details on the data, the declustering method and exploratory analysis concerning stationarity and asymptotic dependence; see also \cite{kee2009, kee2013} for other approaches to flood risk assessment.

The declustering yields $N=428$ supposedly independent events $\g x^{(1)}, \dots, \g x^{(N)}\in \mathbb R^d$. The univariate marginal distributions of these data are
estimated in \cite{asadi2015extremes} by a regionalized extreme value model. We focus on estimation
of the extremal dependence and normalise the data
empirically to standard Pareto marginals. This still guarantees consistent inference of the dependence parameters \cite[e.g.,][]{gen1995, joe2015}. We obtain $n=117$ approximate samples of
$\g Y$ by $\g y^{(h)} = \g x^{(h)}/u$ for all observations with $\|\g x^{(h)}\|_\infty > u$, where
we choose the threshold $u$ as the $90\%$-quantile of the marginal Pareto distribution.

The max-stable Brown--Resnick model in \cite{asadi2015extremes} corresponds to a parametric
family of H\"usler--Reiss Pareto distributions $\{f_{\g Y}(\cdot ; \theta): \theta \in \Omega\}$
at the $31$ gauging stations.
The dependence model is tailor-made for this particular application to river extremes
and uses several covariates such as distance on the river network, catchment sizes and altitudes.
In terms of our new notion of extremal graphical models it is readily checked using the results of Proposition~\ref{prop:CI_HR} that for any parameter value $\theta\in \Omega$ their model does not exhibit conditional independencies.

\begin{figure}[ht]
\centering
{\includegraphics[clip,height=7cm, page=7]{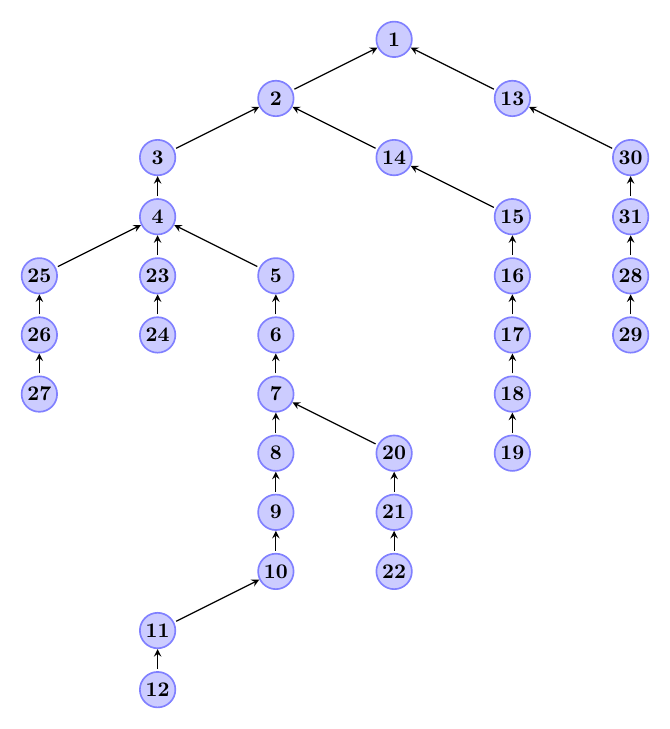}}%
{\includegraphics[clip,height=7cm, page=8]{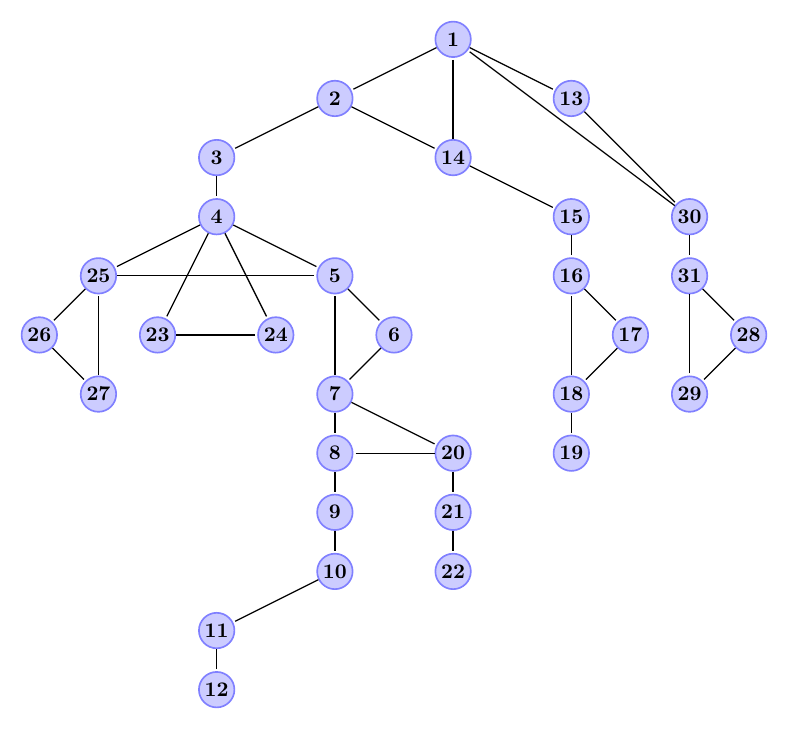}}%
\caption{The tree induced by flow-connections for the $31$ stations in the upper Danube basin (left)
and the estimated graph with the optimal $\AIC$ (right).}
\label{fig_flow}
\end{figure}

We propose a different H\"usler--Reiss model that factorizes according to a sparse graph and does not require any domain knowledge or additional covariates.
In fact, we propose a sequence of models 
$$M^{(l)} = \{f_{\g Y}(\cdot ; \theta^{(l)}): \theta^{(l)} \in \Omega^{(l)}\}, \quad l=1,\dots, L,$$
where $\theta^{(l)} = (\theta^{(l)}_C)_{C\in \mathcal C^{(l)}}$, and $\mathcal C^{(l)}$ is the set
of all cliques of the $l$th extremal graphical model $\mathcal G^{(l)}$ according to which
the model family $M^{(l)}$ factorizes.
As simplest model we take $\mathcal G^{(1)}$ to be the minimum spanning tree within the family of
H\"usler--Reiss models as described in Section~\ref{model_selection}.
Similarly as in the simulation study
in Section~\ref{simulation_study}, we obtain $\mathcal G^{(2)}, \dots, \mathcal G^{(L)}$ by successively adding edges to the tree $\mathcal G^{(1)}$
in a greedy way while restricting the model class to block graphs with cliques of size at most three.
The estimated tree $\mathcal G^{(1)}$ is shown on the left-hand side of Figure~\ref{est_tree} in Appendix~\ref{appendix_tree}. It is very similar to the
tree in Figure~\ref{fig_flow} that corresponds to the tree induced by the flow-connections of the river network. There are however differences, and it is important to note that the flow-connection
tree is not necessarily the optimal tree structure in terms of extreme river discharges. 
Appendix~\ref{appendix_tree} also contains a sensitivity analysis of the tree structure
for different thresholds $u$, and a comparison to a Gaussian tree model fitted to non-extremal data.

Figure \ref{fig_AIC} shows the $\AIC$ values for the different models $M^{(1)}, \dots, M^{(L)}$.
The forward selection is a greedy approach and it does not guarantee to find the optimal
graph. We therefore also initialize the forward selection with the simplest model $\mathcal G^{(1)}$
being the flow-connection tree on the left-hand side of Figure \ref{fig_flow}. This tree
must have a larger $\AIC$ than the minimum spanning tree, but interestingly, the left panel of Figure \ref{fig_AIC}
shows that by adding additional edges the optimal $\AIC$ is better than
the previous optimal $\AIC$. In this particular case, we thus choose the graph
initiated with the flow-connection tree with $9$ additional edges.
In general, a tree structure appears to be too simple for this application. The reason is that only
part of the extremal dependence of discharges at different locations can be explained
by flow-connections. Additional dependence may arise even between flow-unconnected locations
due to proximity of their catchments that are affected by the same spatial precipitation
events. \cite{asadi2015extremes} model this explicitly through a variogram with two parts, one for the dependence on the river network and one for the spatial, meteorological dependence.
The $9$ additional edges of the graphical model on the right-hand side of Figure \ref{fig_flow}, which minimizes the $\AIC$, partly improve the model in terms of this spatial dependence
between flow-unconnected stations, but also strengthen it between some flow-connected
locations. This best graphical model has $39$ edges and an $\AIC$ of $5269.43$. It significantly outperforms the simpler tree models with $30$ edges and the spatial model of \cite{asadi2015extremes}, which has only six parameters but an $\AIC$ of $5291.34$, which is indicated by
the dashed orange line in the left panel of Figure \ref{fig_AIC}.

A popular summary statistic for extremal dependence between $Y_i$ and $Y_j$, $i,j\in V$,
is the tail correlation \citep[cf.,][]{col1999}, which can be expressed as $\chi_{ij} = 2 - \Lambda_{ij}(1,1)$.
The centre and right panels of Figure \ref{fig_AIC} compare empirical estimates of these statistics for all pairs of stations with
those implied by the fitted models. In terms of this bivariate summary, both models seem
to fit the data well, even though the graphical model seems to be slightly less biased
than the spatial model.
There are also versions of $\chi$ that assess how a model captures the higher-order
extremal dependence structure. In Figure~\ref{fig_trivariate} in Appendix~\ref{appendix_trivariate} we compare the trivariate empirical
$\chi$ coefficients with those implied from the fitted spatial and graphical model.
Both models fit well the trivariate dependence, again with a slightly lower bias of the graphical model.

In this application we have only considered block graphs, which are particularly convenient in terms of statistical inference as seen in the previous sections. In general it should be assessed whether this sparse model class is justified for the data. In our case, the bivariate and trivariate $\chi$ coefficients indicate that block graphs are flexible enough to capture
the extremal dependence structure of the river data. This is further supported by the fact that
the AIC curve in Figure~\ref{fig_AIC} attains its minimum even before the maximal number of edges is added in this model class. It is an important question for future research how extremal graphical models with more complicated structures can be
estimated.

\begin{figure}[ht]
  \centering
  \hspace*{-3em}
{\includegraphics[height=6cm]{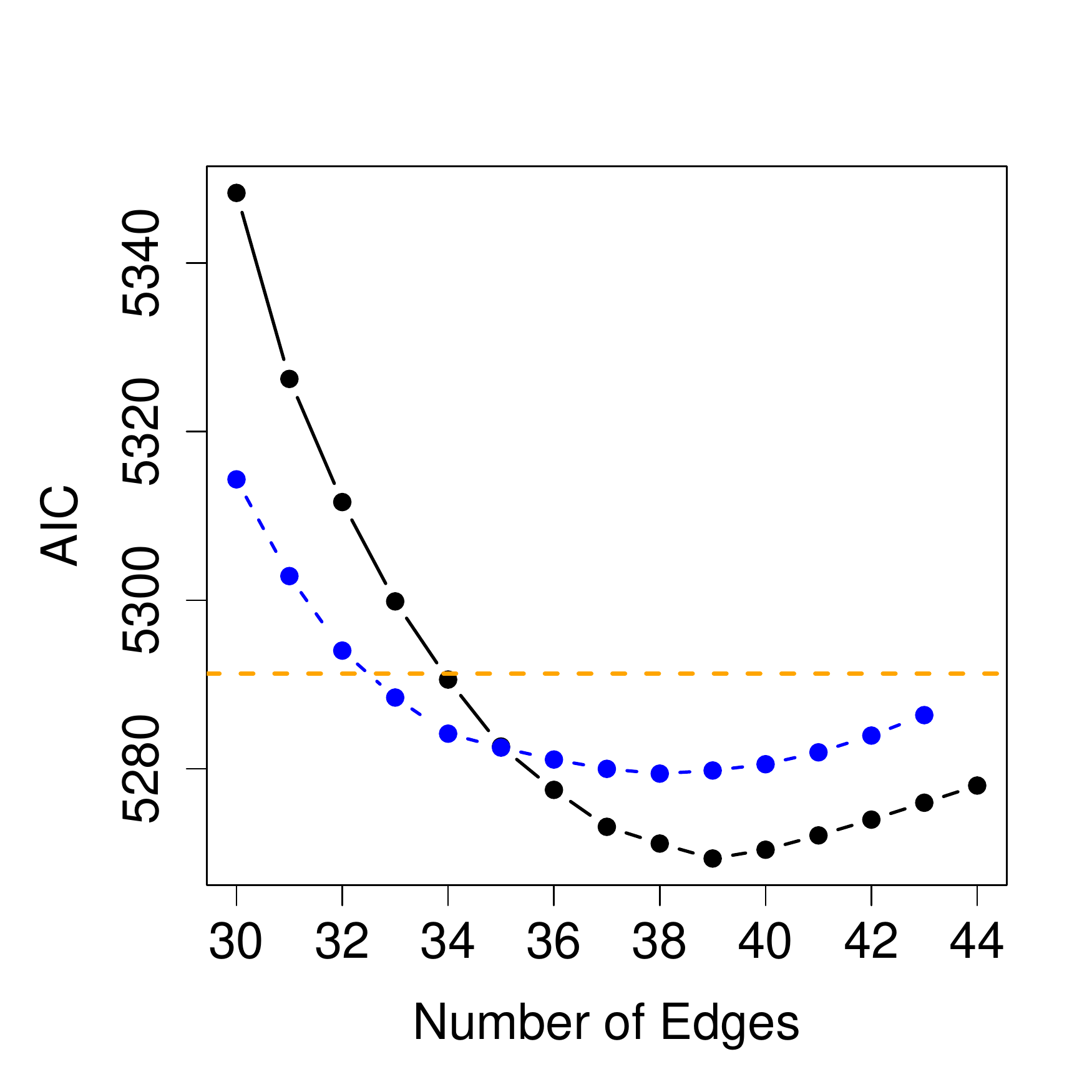}}%
{\includegraphics[height=6cm]{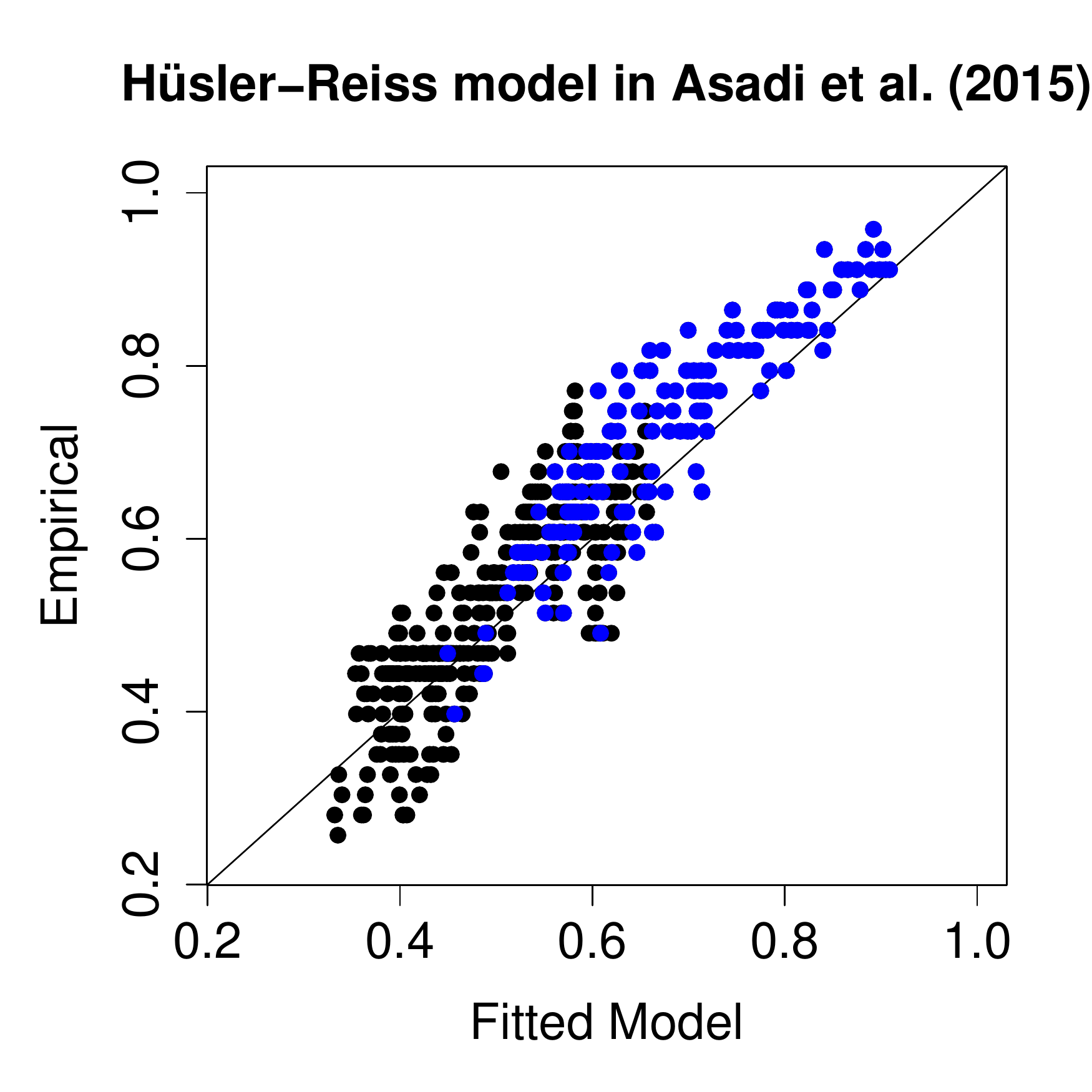}}%
{\includegraphics[height=6cm]{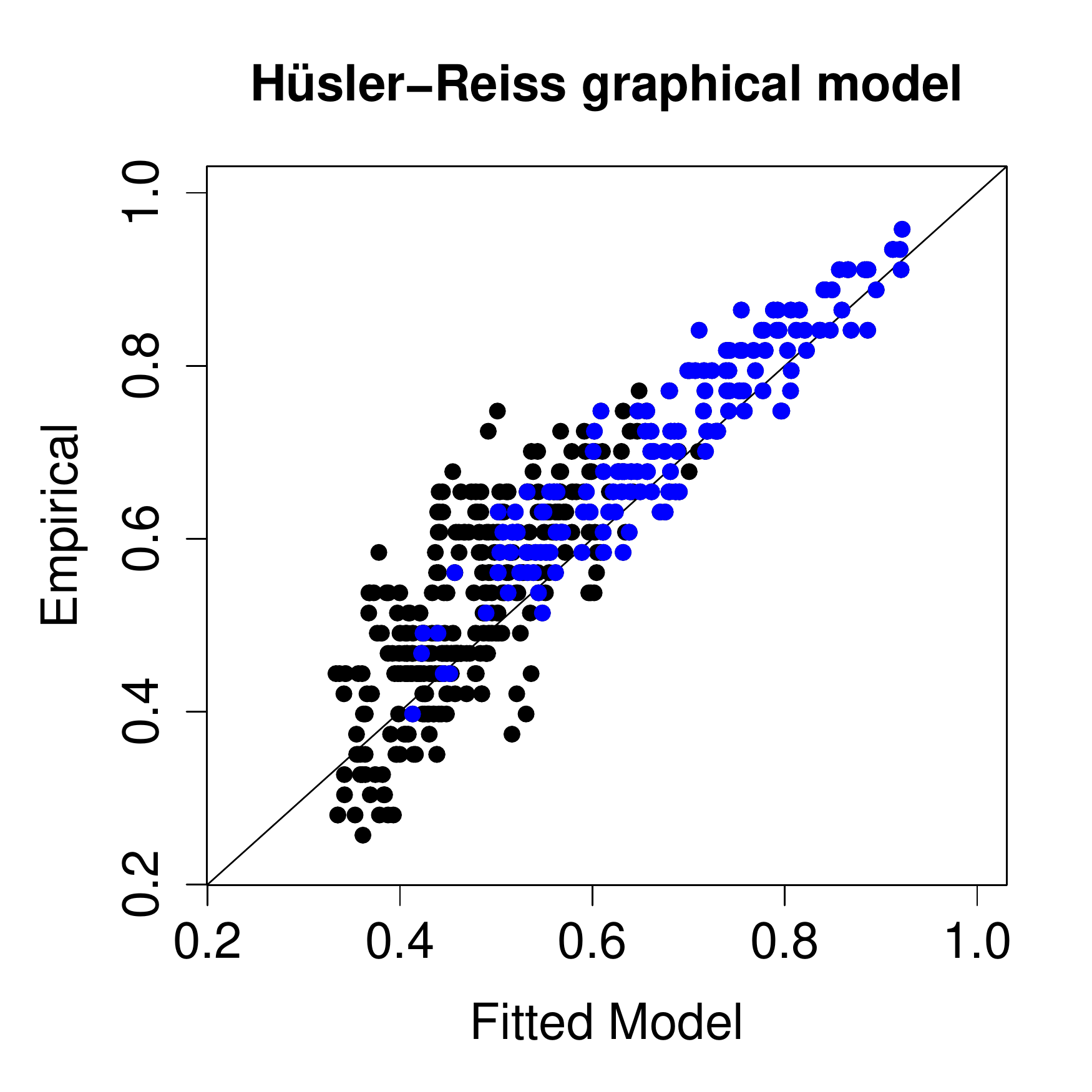}}%
\caption{Left: $\AIC$ values for extremal graphical models for the Danube data set with an increasing number of edges, starting from the minimum spanning tree (blue line) and the flow-connection tree (black line); the horizontal orange line is the $\AIC$ of the spatial model in \cite{asadi2015extremes}.
  Centre and right: empirically estimated $\chi$ coefficients against those implied by the fitted spatial and graphical model minimizing the $\AIC$, respectively; blue points correspond to flow-connected stations.}
\label{fig_AIC}
\end{figure}

\section{Discussion}
\label{sed:disc}

The conditional independence relation $\perpe$ introduced in this paper
is natural for a multivariate Pareto distribution $\g Y$ as it explains the factorization
of its density $\g f_{\g Y}$ into lower-dimensional marginals (cf., Theorem~\ref{thm:equivFactorHomog}).
This establishes a link of extreme value statistics to the broad field of graphical models,
and it opens the door to define sparsity and to perform structure learning for tail distributions.
In this work we have studied the probabilistic structure and statistical inference for
some important models, with the main purpose of modelling the extremal dependence structure.
Many subsequent research directions are possible.
Directed acyclic graphs as in \cite{gis2018}
for max-linear models may be formulated in our setting and would yield different factorizations
than for undirected graphs, and this would form the basis to extend work
on causal inference for extremes \citep{nav2018,mha2019,gne2019} to continuous extreme value distributions.
The models in this paper are well-suited for asymptotic dependence. Another line of research
focuses on multivariate tail models under asymptotic independence \citep{led1997,HeffernanTawn2004,Wadsworthetal2017}. Conditional independence and graphical models have not been
studied in this framework, except for the special case of Markov chains \citep{kul2015,pap2017}.

Conditional independence for $\g Y$ does not carry over to factorization of the density of the associated max-stable distribution $\g Z$. By Proposition \ref{prop_equiv}, the conditional independence relation~$\perpe$
does however imply the factorization of the exponent measure density $\lambda$ of $\g Z$, which is the key object
in simulation \citep{dom2016} and full likelihood estimation \citep{thi2015, dom2016a, hus2019} of max-stable processes. Thus, sparsity in our notion for multivariate Pareto distributions also facilitates inferential tasks for max-stable distributions, a fact that has been briefly discussed for simulation in Section \ref{simulation} but deserves further investigation. 

The application to flood risk assessment is just one illustrative example. Unlike spatial models, extremal graphical models can be applied to multivariate problems without domain knowledge, as for instance
in financial or insurance applications. The ability to learn underlying structures in a data-driven
way has also great practical potential for exploratory analysis and data visualization.
In ongoing research we investigate efficient learning of extremal tree structures and,  in the case of H\"usler--Reiss distributions, of more general graphs based on $\ell_1$-regularization.

\section*{Acknowledgments}
We thank Robin J.~Evans and Nicola Gnecco for helpful discussions. We are grateful to the editorial team and the referees for knowledgeable comments that improved the paper.
Financial support by the Swiss National Science Foundation (S.~Engelke) and by the Berrow Foundation (A.~S.~Hitz) is gratefully acknowledged. The paper was completed while S.~Engelke was a visitor at the Department of Statistical Sciences, University of Toronto.

\appendix

\section*{Appendix}
\addcontentsline{toc}{section}{Appendices}
\renewcommand{\thesubsection}{\Alph{subsection}}
\label{appendix}

\subsection{Definitions for graphical models}
\label{def_graphs}

Let $\mathcal G = (V, E)$ be an undirected graph with node set $V = \{1,\dots, d\}$ and
edge set $E \subset V\times V$; see Section~\ref{graph_models}. We define
the notion decompositions and decomposability for the graph $\mathcal G$ \citep[cf.,][Definition 2.1]{Lauritzen}.
\begin{definition}
A triplet $(A,B,C)$ of disjoints subsets of $V$ is said to form a {decomposition} of $\mathcal G$ into the components $\mathcal G_{A\cup B}$ and $\mathcal G_{B\cup C}$ if $V=A\cup B\cup C$ and
\begin{itemize}
\item $B$ separates $A$ from $C$ (i.e., every path from $A$ to $C$ intersects $B$);
\item $B$ is a complete subset.
\end{itemize}
The decomposition is called {proper} if $A$ and $C$ are both non-empty. A graph $\mathcal G$ is {decomposable} if it is complete or if there exists a proper decomposition $(A,B,C)$ into decomposable subgraphs $\mathcal G_{A\cup B}$ and $\mathcal G_{B\cup C}.$ Decomposable graphs are also known as triangulated or chordal graphs.

For instance, $(\{1, 2,3,4,5\},\{4,5\},\{4,5,6\})$ is a proper decomposition of the decomposable graph in Figure~\ref{fig:tikz_simpleGraph}. 
 \end{definition}
 For a connected, decomposable graph $\mathcal G$, we can order the set of the cliques $\mathcal C = \{C_1,\dots, C_m\}$ such that for all $i=2,\dots, m$, 
 \begin{align}\label{rip}
   D_i := C_i \cap \bigcup_{j=1}^{i-1}  C_j \subset C_k \text{ for some } k< i,
   \end{align}
 a condition called the running intersection property; cf., \citet[][Chapter 2]{Lauritzen} and \cite{gre2013}.
 The sets $D_i$, $i=2,\dots, m$, are called separators of the graph, and both $\mathcal C$
 and the collection of separators $\mathcal D = \{D_2,\dots, D_m\}$ are uniquely determined up to different orderings.
 The separators may not all be distinct, and we say that $\mathcal D$ is a multiset.
 A possible enumeration of cliques and separators for the graph in Figure \ref{fig:tikz_simpleGraph} that satisfies the running intersection property  is 
$$\mathcal C=(\{1,2\},\{2,3,4,5\},\{4,5,6\}), \quad \mathcal D=(\{2\},\{4,5\}).$$
  From \eqref{rip} we note that the clique $C_m$ intersects the other cliques only in $D_m$. Consider the connected, decomposable subgraph $\mathcal G_{m-1}$ of $\mathcal G$ with node set $V_{m-1}= V \setminus  (C_m\setminus D_m)$ and corresponding induced edge set.
 The property \eqref{rip} then holds for $\mathcal G_{m-1}$, which has one clique less. Continuing this process, we note that each $C_j$ intersects the subgraph $\mathcal G_{j}$ only in $D_j$, $j=2,\dots, m$,
 and $\mathcal G_1$ with nodes $V_1 = C_1$ is complete.

\begin{figure}[h]
    \centering
        \begin{subfigure}{0.49\textwidth}
        \includegraphics[width=\textwidth, page=5]{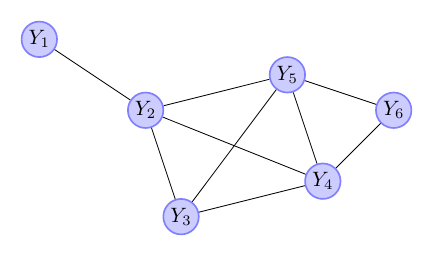}
    \end{subfigure}
            \caption[A decomposable graph]{A decomposable graph with set of nodes $V=\{1,\dots,6\}.$ The cliques of the graph are $\{1,2\},$ $\{2,3,4,5\}$ and $\{4,5,6\}.$ The separators are $\{2\}$ and $\{4,5\}.$ 
    }\label{fig:tikz_simpleGraph}
\end{figure}

\subsection{Link between variogram and covariance matrices}
\label{link_vario}
For $k\in V = \{1,\dots, d\}$, we denote  by $\mathcal P_{d-1}^k$ the set of all strictly positive definite covariance matrices $\Sigma^{(k)} \subset \mathbb R^{(d-1)\times (d-1)}$
indexed by $V \setminus \{k\}$.
On the other hand, the space of strictly conditionally negative definite $d\times d$ matrices is denoted by
\begin{align*}
  \mathcal D_d = \bigg\{\Gamma \in [0,\infty)^{d\times d}:  \g a^\top \Gamma \g a < 0 \text{ for all } & \g a \in \mathbb R^d\setminus \{0\} \text{ with } \sum_{i\in V} a_i = 0,\\
    & \Gamma_{ii} = 0, \Gamma_{ij}= \Gamma_{ji} \text{ for all } i,j\in V  \bigg\}. 
\end{align*}

\begin{lemma}
  For any $k\in V$, there is a bijection $\varphi_k:\mathcal D_d\to \mathcal P_{d-1}^k$ given by
\begin{equation}
\begin{aligned}\label{phik}
    \varphi_k:&\, \Gamma \mapsto \frac{1}{2} \{\Gamma_{ik}+\Gamma_{jk}-\Gamma_{ij}\}_{i,j\neq k},\\
    \varphi_k^{-1}:&\,  \Sigma^{(k)} \mapsto \einsfun \diag(\tilde\Sigma^{(k)})^\top +  \diag(\tilde\Sigma^{(k)})\einsfun^\top - 2 \tilde\Sigma^{(k)},
\end{aligned}
\end{equation}
where $\tilde\Sigma^{(k)}$ is the $d\times d$ matrix that coincides with $\Sigma^{(k)}$ for $i,j\neq k$ and that has zeros in the $k$th column and row.
\end{lemma}
\begin{proof}
  It is easy to check that the mappings are their mutual inverses.
  To see that the strict positive definiteness of $\Sigma^{(k)}$ is equivalent
  to the strict conditionally negative definiteness of $\Gamma$, we observe
  for any $\g a_{\setminus k} \in \mathbb R^{d-1}\setminus \{\g 0\}$ and $a_k = - \sum_{i\neq k} a_i$
  $$\g a_{\setminus k}^\top \Sigma^{(k)}\g a_{\setminus k} =  \frac{1}{2}\sum_{i,j\neq k} a_ia_j \left(\Gamma_{ik}+\Gamma_{jk}-\Gamma_{ij}\right) = -\sum_{i\neq k} a_i a_k \Gamma_{ik} - \frac{1}{2}\sum_{i,j\neq k} a_ia_j\Gamma_{ij} = - \g a^\top \Gamma\g a,$$
  using the fact that $\Gamma$ is symmetric and  $\Gamma_{ii} = 0$ for all $i\in V$. The assertion then follows; see also the proof of Lemma 3.2.1 in \cite{ber1984}.
\end{proof}

\subsection{H\"usler--Reiss densities on decomposable graphs}
\label{app_decomp}
\begin{corollary}\label{cor_HR_dens} 
  Let $\mathcal G = (V,E)$ be a decomposable and connected graph, and suppose that 
  $\g Y$ is a H\"usler--Reiss Pareto distribution that satisfies the pairwise Markov property
  \begin{align*}
    Y_i \perpe Y_j \mid \g Y_{\setminus \{i,j\}} \quad \text{if } (i,j)\notin E.
  \end{align*}
  Then the density of $\g Y$ factorizes according to $\mathcal G$ into lower-dimensional H\"usler--Reiss densities, that is,
  \begin{align*}
    f_{\g Y}(\g y)= \frac{y_{k_{1}}^{-2} \prod_{j\neq k_{1}} y_j^{-1}}{\Lambda(\mathbf 1)} {\prod_{i=1}^m \phi_{|C_i|-1}\left\{ \log( \g y_{C_i\setminus\{k_i\}} / y_{k_i}) + \Gamma_{C_i\setminus\{k_i\}, k_i}; \Sigma^{(k_i)}_{C_i}\right\} 
    \over \prod_{i=1}^{m-1} \phi_{|D_i|-1}\left\{ \log( \g y_{D_i\setminus\{k_i\}} / y_{k_i}) + \Gamma_{D_i\setminus\{k_i\}, k_i}; \Sigma^{(k_i)}_{D_i}\right\} } ,\quad \g y\in \mathcal L,
  \end{align*}
  where the sequences of cliques $\{C_1,\dots,C_m\}$ and separator sets $\{D_2,\dots, D_m\}$
  have the running intersection property \eqref{rip}, and $k_i\in D_i$, $i=2,\dots, m$, $k_1\in C_1$.
\end{corollary}

\begin{proof}
  Theorem \ref{thm:equivFactorHomog} and Proposition \ref{prop:CI_HR} yield the factorization. It remains to show that the factors in front of the normal densities simplify to 
  $y_{k_{m-1}}^{-2} \prod_{i\neq k_{m-1}} y_i^{-1}$. Indeed, since we choose $k_i \in D_i \subset C_i$, $i=2,\dots, m$, the ratio $\lambda_{C_i}(\g y_{C_i}) / \lambda_{D_i}(\g y_{D_i})$ contributes the factor $y_j^{-1}$ for all $j\in C_i\setminus D_i$, and each such $j$ appears exactly once. For $i=1$, the contribution of $\lambda_{C_1}(\g y_{C_1})$ is $y_{k_{1}}^{-2} \prod_{i \in C_1 \setminus \{k_{1}\}} y_i^{-1}$.
\end{proof}

\subsection{Minimum spanning tree for the Danube river}
\label{appendix_tree}

The left-hand side of Figure~\ref{est_tree} shows the estimated H\"usler--Reiss minimum spanning tree
for the Danube data in Section~\ref{sec:appl} for a threshold $u$ chosen as the $90\%$-quantile of the marginal Pareto distribution. In order to assess the sensitivity of the
tree structure with respect to the threshold choice, we estimate the minimum spanning tree
for thresholds $u$ corresponding to a range of different quantiles. The similarity of
these trees in terms of the number of identical edges compared to the $90\%$-quantile tree are shown in Figure~\ref{treediff}. One can see that there is some variation of the tree structure for different thresholds, but that most of the $30$ edges are fairly stable throughout a wide range of thresholds. As a comparison, the right-hand side of Figure~\ref{est_tree} shows the Gaussian minimum spanning tree fitted to all log-transformed data, using $\log(1-\rho_{ij}^2)$ as distances in~\eqref{Tmin}, where $\rho_{ij}$ is the correlation coefficient between nodes $i,j\in V$. The Gaussian tree, a model for non-extremal data, is similar to the H\"usler--Reiss tree, a model for extreme flooding, but there are also some differences. For instance, for the extremal data the ordering of the stations 16 to 19 seems to be less important since large discharges affect all at the same time. This is confirmed by the fact that when the H\"usler--Reiss tree is extended to a block graph, then additional edges are introduced between these stations.

\begin{figure}[ht]
  \centering  
{\includegraphics[clip,height=7cm, page=9]{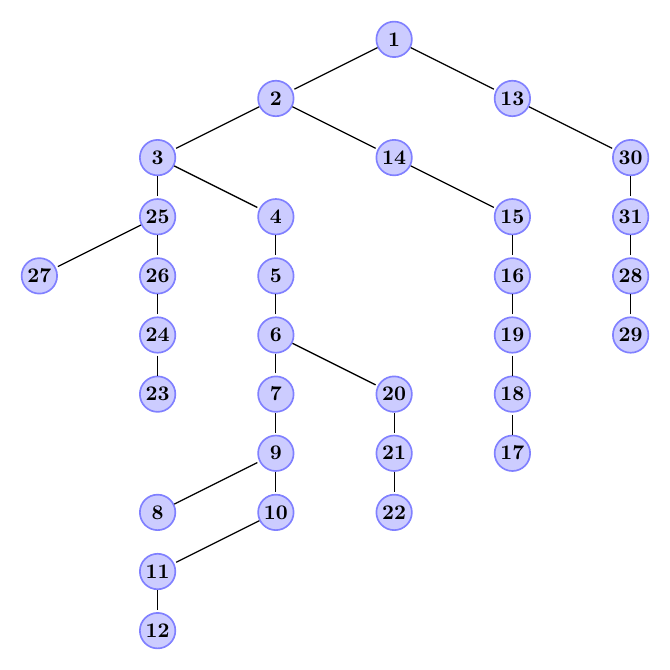}}%
{\includegraphics[clip,height=7cm, page=11]{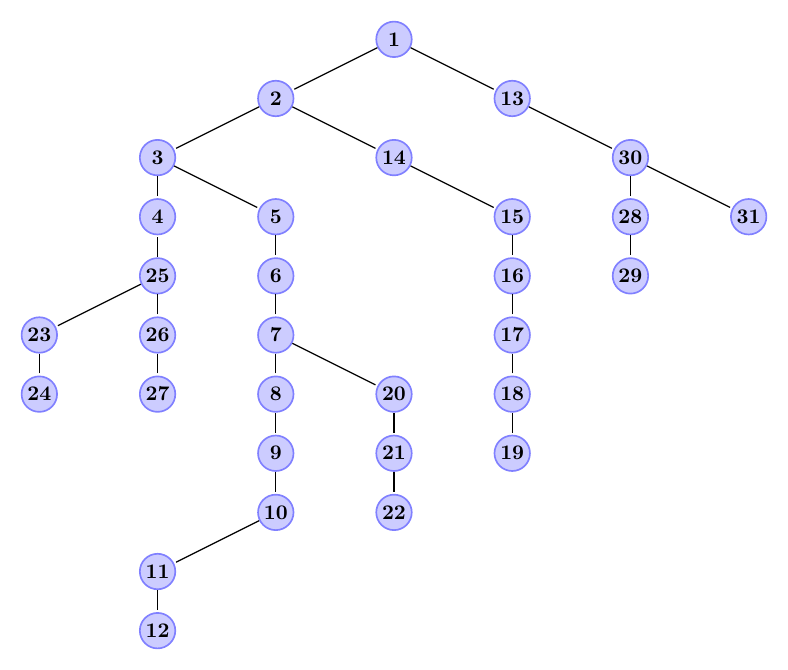}}%
\caption{Estimated H\"usler--Reiss minimum spanning tree
  for the Danube data with $90\%$-quantile threshold (left)
  and Gaussian minimum spanning tree using all log-transformed data (right).}
\label{est_tree}
\end{figure}

\begin{figure}[ht]
\centering
{\includegraphics[height=7cm]{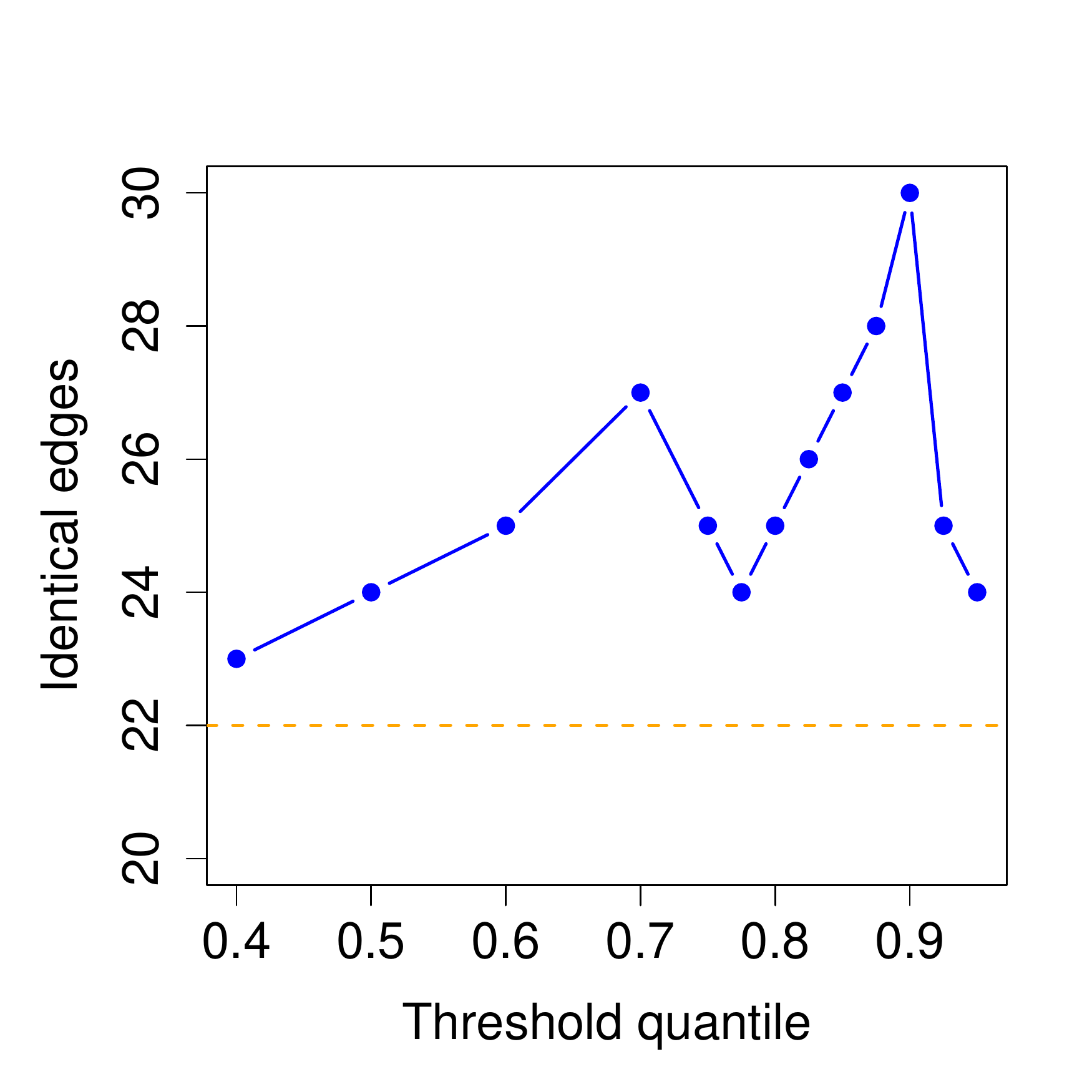}}%
\caption{For estimated minimum spanning trees corresponding to different threshold quantiles, the blue line shows the number of edges that are identical to $90\%$-quantile tree. The horizontal orange line is the number of identical edges for a Gaussian minimum spanning tree using all log-transformed data.}
\label{treediff}
\end{figure}

\subsection{Trivariate $\chi$ coefficients}
\label{appendix_trivariate}

Figure~\ref{fig_trivariate} shows the empircal estimates of the trivariate coefficients
$$\chi_{ijk} = 3 - \Lambda_{ij}(1,1) - \Lambda_{ik}(1,1) - \Lambda_{jk}(1,1) + \Lambda_{ijk}(1,1,1), \quad i,j,k\in V,$$
against those implied by the fitted spatial model in \cite{asadi2015extremes} and our graphical model minimizing the $\AIC$.

\begin{figure}[ht]
  \centering
  \hspace*{-3em}
{\includegraphics[height=6cm]{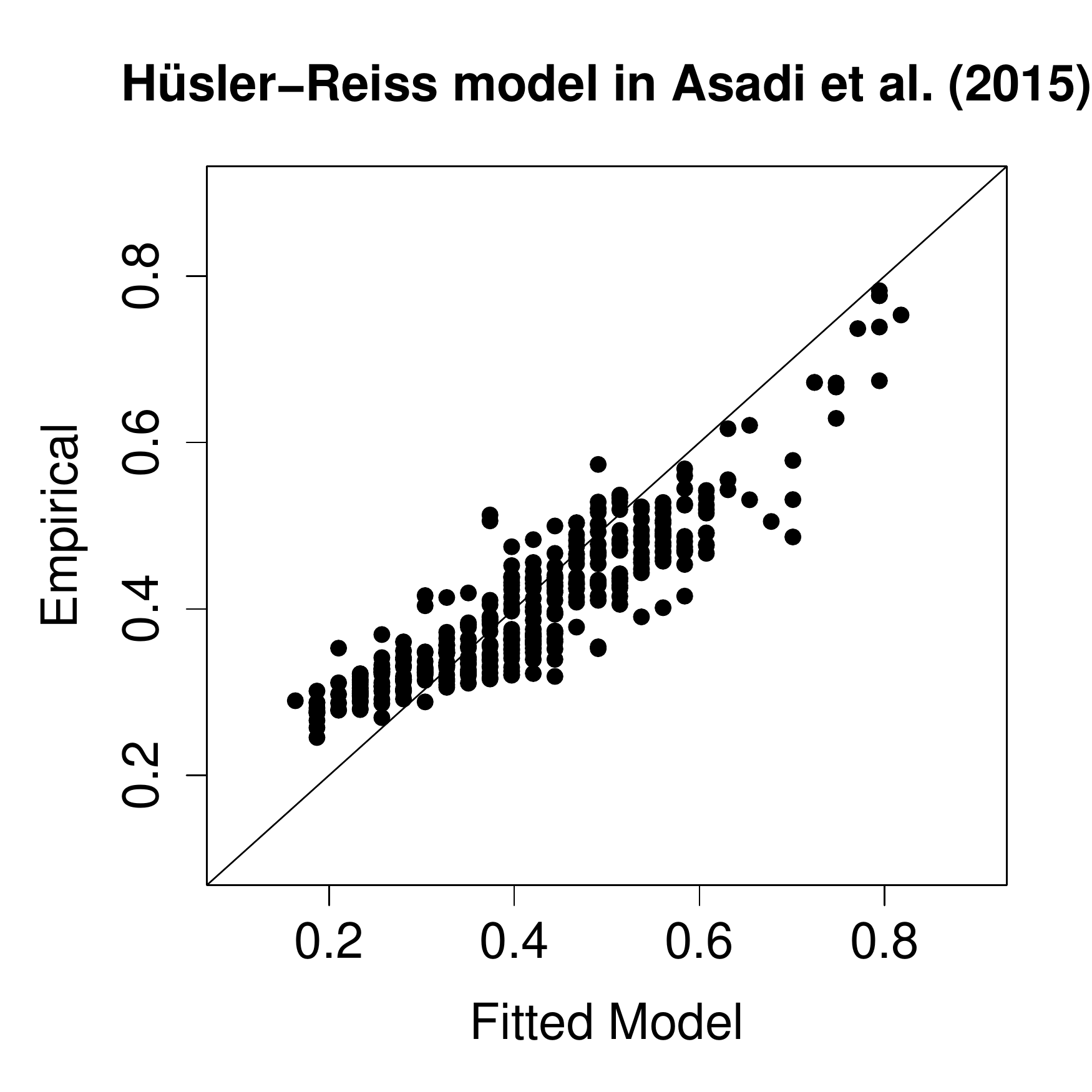}}%
{\includegraphics[height=6cm]{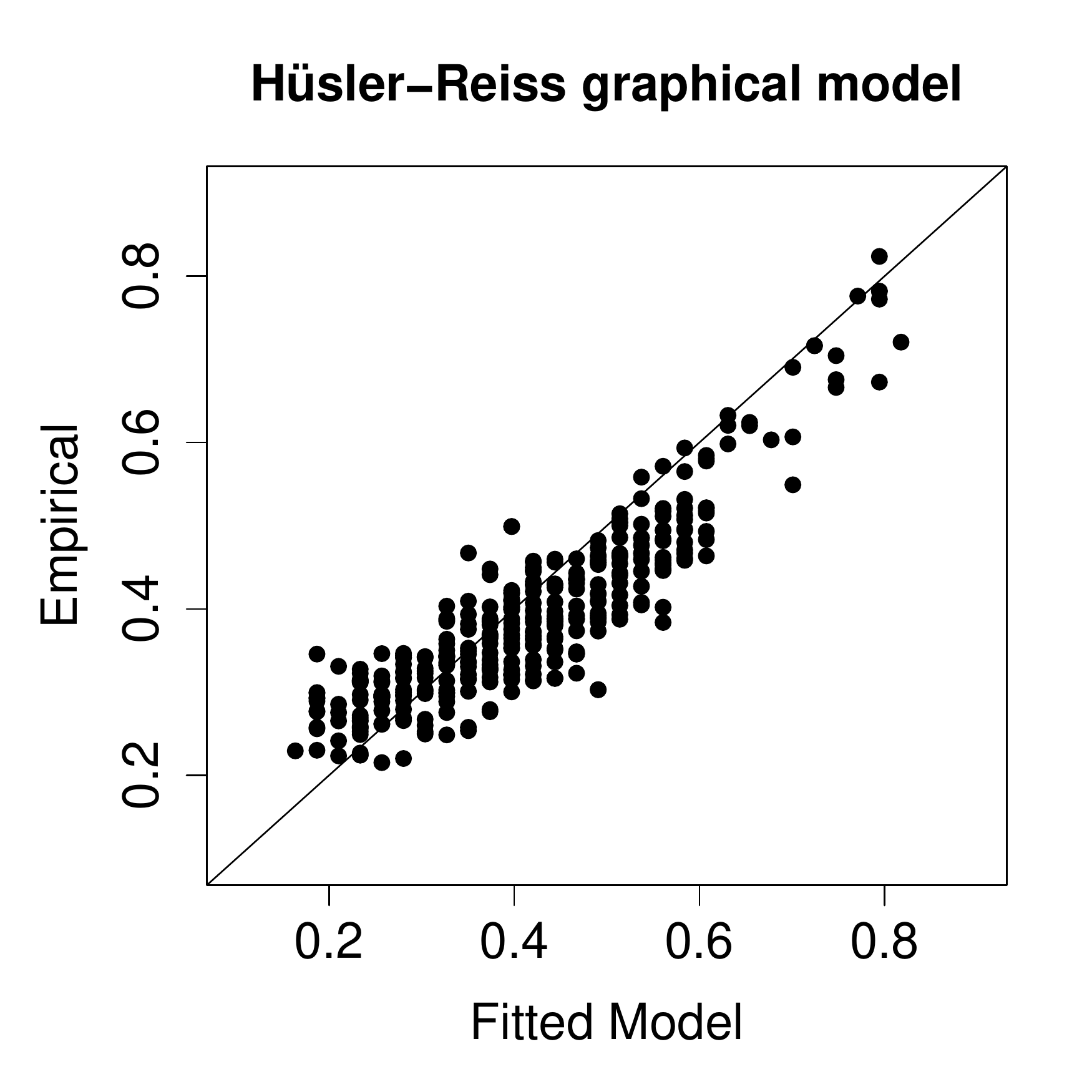}}%
\caption{Empirical estimates of the trivariate coefficients $\chi_{ijk}$, $i,j,k\in V$, against those implied by the fitted spatial model in \cite{asadi2015extremes} and our graphical model minimizing the $\AIC$; only coefficients for 400 randomly selected triplets are shown.}
\label{fig_trivariate}
\end{figure}

\subsection{Proofs}
\label{appendix_proofs}

\begin{proof}[of Proposition \ref{prop_equiv}]
  The implication $\eqref{eq:citail3} \Rightarrow (i)$ is trivial. For $(i) \Rightarrow (ii)$ let $k \in B$ and suppose that~\eqref{ci_weak} holds, that is, 
  \begin{align*}
     f^k(\g y) =  {f^k_{A\cup B}(\g y_{A\cup B}) f^k_{B\cup C}(\g y_{B\cup C}) \over  f^k_{B}(\g y_B)}, \quad \g y\in \mathcal L^k.
  \end{align*}
  For any $\g y \in \mathcal L$ choose $0 < t < \min(y_k,1)$, i.e., $\g y / t \in \mathcal L^k$, and observe
\begin{align*}
  \lambda(\g y) &=  t^{-(d+1)} f^k(\g y/ t) \\
  &=    t^{-(d+1)} {f^k_{A\cup B}(\g y_{A\cup B}/t) f^k_{B\cup C}(\g y_{B\cup C}/t) \over  f^k_{B}(\g y_B /t)} \\
  & =    t^{-(d+1)} {\lambda_{A\cup B}(\g y_{A\cup B}/t) \lambda_{B\cup C}(\g y_{B\cup C}/t) \over  \lambda_{B}(\g y_B /t)} \\
  & =    {\lambda_{A\cup B}(\g y_{A\cup B}) \lambda_{B\cup C}(\g y_{B\cup C}) \over  \lambda_{B}(\g y_B )}, 
\end{align*}
using the homogeneity of the $\lambda_I$, and the fact that $f_I^k(\g y_I / t) = \lambda_I(\g y_I/ t)$ for any $I \subset V$ with $k\in I$. Note that for this argument it is crucial that
$k$ is in an element of all three sets $B$, $A\cup B$ and $B\cup C$.

For $(ii) \Rightarrow~\eqref{eq:citail3}$ suppose that the factorization~\eqref{lambda_fac} of $\lambda$ holds, and let $k\in V$. For all $\g y \in \mathcal L^k$
\begin{align*}
  f^k(\g y) =    {\lambda_{A\cup B}(\g y_{A\cup B}) \lambda_{B\cup C}(\g y_{B\cup C}) \over  \lambda_{B}(\g y_B )} = g(\g y_{A\cup B}) h(\g y_{B\cup C}),
\end{align*}
for suitable functions $g$ and $h$, implying the required conditional independence of $f^k$ \citep[cf.,][Chapter 3]{Lauritzen}.
This shows that condition~\eqref{eq:citail3} indeed holds and thus $\g Y_A\perpe \g Y_C\mid \g Y_B$.
\end{proof}

\begin{proof}[of Theorem \ref{thm:equivFactorHomog}]
  We start by proving that if $\g Y$ satisfies the pairwise Markov property relative to $\mathcal G$, then the graph $\mathcal G$ is necessarily connected. 
  Indeed, suppose $V$ can be split into non-empty, disjoint subsets $V_1,V_2 \subset V$
  such that for $(i,j)\in E$ it holds either $i,j \in V_1$ or $i,j \in V_2$.
  For an arbitrary $k\in V$, by assumption, the pairwise Markov property relative to $\mathcal G$  is satisfied for $f^k$ on $\mathcal L^k$ and the classical Hammersley--Clifford theorem implies the global Markov property for $f^k$, and in particular  
$$f^k(\g y) =  f_{V_1}^k(\g y_{V_1}) f_{V_2}^k(\g y_{V_2}), \quad \g y \in \mathcal L^k.$$
The discussion after Proposition~\ref{prop_equiv} shows that such as factorization contradicts integrability of the
multivariate Pareto density, and therefore the graph has to be connected.

We now show that $(i) \Rightarrow (iii)$. 
The pairwise Markov property of $f^k$ relative to $\mathcal G$
implies by the classical Hammersley--Clifford theorem that 
$$f^k(\g y) =  {\prod_{C\in\mathcal C} f_C^k(\g y_C) \over \prod_{D\in\mathcal D} f_D^k(\g y_D) }, \quad \g y \in \mathcal L^k.$$
This representation is not of direct use since it cannot be extended to $f_{\g Y}$ on the whole space $\mathcal L$, since all $f_I^k$ with $k\notin I$ are not homogeneous. The result however tells us that $\g Y^k$ also satisfies the global Markov property on $\mathcal L^k$ relative to $\mathcal G$, as defined in Section \ref{graph_models}.
 The running intersection property implies that $D_m$ separates $C_m\setminus D_m$ from $(C_1\cup \dots \cup C_{m-1})\setminus D_m$. Choose $k\in D_m$, then the global Markov property for $\g Y^k$ yields
\begin{align*}
  f^k(\g y) = {f_{C_m}^k(\g y_{C_m}) f_{C_1\cup \dots \cup C_{m-1}}^k(\g y_{C_1\cup \dots \cup C_{m-1}}) \over f_{D_m}^k(\g y_{D_m})} = {\lambda_{C_m}(\g y_{C_m}) \lambda_{C_1\cup \dots \cup C_{m-1}}(\g y_{C_1\cup \dots \cup C_{m-1}}) \over \lambda_{D_m}(\g y_{D_m})} ,  \quad  \g y \in \mathcal L^k,
\end{align*}
where the second equality holds since $k\in D_m$, and $D_m$ is a subset of both $C_m$ and $C_1\cup \dots\cup C_{m-1}$.
By a homogeneity argument similar to the proof of Proposition~\ref{prop_equiv},
this factorization extends to $\lambda$ on the whole space $\mathcal L$, that is,
\begin{align*}
  \lambda(\g y) =  {\lambda_{C_m}(\g y_{C_m}) \lambda_{C_1\cup \dots \cup C_{m-1}}(\g y_{C_1\cup \dots \cup C_{m-1}}) \over \lambda_{D_m}(\g y_{D_m})} ,  \quad  \g y \in \mathcal L.
\end{align*}
It remains to decompose $ \lambda_{C_1\cup \dots \cup C_{m-1}}$ in the same manner.
To this end, choose a new $k\in D_{m-1}$ and note that 
\begin{align*}
  f_{C_1\cup \dots \cup C_{m-1}}^k(\g y_{C_1\cup \dots \cup C_{m-1}}) 
 =\int_{[0,\infty)^{|C_m\setminus D_m|}} {\prod_{C\in\mathcal C} f_C^k(\g y_C) \over \prod_{D\in\mathcal D} f_D^k(\g y_D) } \mathrm d\g y_{C_m\setminus D_m} =  {\prod_{C\neq C_m} f_C^k(\g y_C)  \over \prod_{D\neq D_m} f_D^k(\g y_D) },
\end{align*}
and therefore satisfies the global Markov property relative to the subgraph
induced on $C_1\cup \dots \cup C_{m-1}$. Since $f_{C_1\cup \dots \cup C_{m-1}}^k = \lambda_{C_1\cup \dots \cup C_{m-1}}$ on $\mathcal L^k$, applying successively the same reasoning as before yields the factorization of $\lambda$ that directly implies the representation in~\eqref{eq:producthc} for~$f_{\g Y}$.

In order to show that $(iii) \Rightarrow (ii)$,
we only need to verify that $\g Y^k$ satisfies the global
Markov property on $\mathcal L^k$ for any $k\in V$.
For disjoint sets $A,B,C\subset V$ such that $B$ separates $A$ from $C$,
the factorization~\eqref{eq:producthc} entails that
\begin{align*}
  f^k(\g y) =  \Lambda(\g 1) f_{\g Y}(\g y) =  g(\g y_{A\cup B}) h(\g y_{B\cup C}),
\end{align*}
for suitable functions $g$ and $h$, and thus $\g Y^k_A\ci \g Y^k_C\mid \g Y^k_B$.

The implication $(ii) \Rightarrow (i)$ holds trivially.
\end{proof}

\begin{proof}[of Corollary \ref{cor_construction}]
It is easy to check that $\lambda$ and $f_{\g Y}$ are homogeneous of order $-(d+1)$ on $\mathcal L$. Let $\{C_1,\dots, C_m\}$ and $\{D_2, \dots, D_m\}$ be the sequences of cliques and separators with the running intersection property \eqref{rip}. Sequential integration of the function $f_{\g Y}$ on $C_m\setminus D_m,\dots, C_2\setminus D_2,$ together with the consistency constraint yields that it defines in fact a probability density. Theorem \ref{thm:equivFactorHomog} implies that the corresponding distribution on $\mathcal L$ satisfies the Markov property relative to $\mathcal G$.
\end{proof}

\begin{proof}[of Proposition \ref{prop_tree}]
  The density of the random vector on the right-hand side of~\eqref{tree_rep} 
  is 
  $$ y_k^{-2} \prod_{e=(i,j)\in E^k} y_i^{-1}f_{U_e}(y_j/y_i) = y_k^{-2} \frac{\prod_{(i,j)\in E^k} \lambda_{ij}(y_i,y_j)}{\prod_{(i,j)\in E^k} y_i^{-2}} =  \prod_{\{i,j\}\in E} {\lambda_{ij}(y_i,y_j) \over y_i^{-2} y_j^{-2}} \prod_{i\in V} y_i^{-2},$$
  where we used~\eqref{dens_biv} for the first equation, and the fact that
  each node $i\in V\setminus\{k\}$ has exactly one incoming arrow, and the $k$th node has
  no incoming arrows. On the other hand, we recall that the density of $\g Y^k$ is $\lambda(\g y) = \Lambda(\g 1) f_{\g Y}(\g y)$, which factorizes with respect to the tree $\mathcal T$. Comparing the above density with~\eqref{tree_fact} yields the result.
\end{proof}

\begin{proof}[of Lemma \ref{cor_HR}]
  Without losing generality, we may and do assume that $k'=1$ and $k=2$.
  Let the vector $\g W^1 = (0, W_2^1,\dots, W_d^1)$ have a centred normal distribution with covariance matrix $\Sigma = \{\sigma_{ij}\}= \tilde \Sigma^{(1)}$, such that 
    $$ \Sigma^{(1)} = \Sigma_{\setminus\{1\}} =
     \begin{bmatrix}
      \sigma_{22} &  \Sigma_{2, \setminus \{1,2\}} \\
       \Sigma_{\setminus \{1,2\},2} &  \Sigma_{\setminus \{1,2\}}
    \end{bmatrix}.    
   $$
   The precision matrix is obtained by blockwise inversion as
   $$ \Theta^{(1)} = 
     \begin{bmatrix}
      \sigma_{22}^{-1} + \sigma_{22}^{-2} \Sigma_{2, \setminus \{1,2\}} S^{-1}\Sigma_{\setminus \{1,2\},2}  & -  \sigma_{22}^{-1}\Sigma_{2, \setminus \{1,2\}} S^{-1}\\
       -  \sigma_{22}^{-1} S^{-1} \Sigma_{\setminus \{1,2\},2} & S^{-1}
    \end{bmatrix},  
   $$
   where $S = \Sigma_{\setminus \{1,2\}} - \sigma_{22}^{-1} \Sigma_{\setminus \{1,2\},2}\Sigma_{2, \setminus \{1,2\}}$ is the Schur complement of upper left block $\sigma_{22}$ in the matrix $\Sigma^{(1)}$. The random vector $\g W^1$ can be transformed into
  $$\g W^2 = (-W_2^1, 0 , W_3^1 - W_2^1, \dots, W_d^1 - W_2^1),$$
  which is readily verified to have centred normal distribution with covariance matrix $\tilde \Sigma^{(2)}$. On the other hand, we may write the covariance matrix $\Sigma^{(2)}$
  of $(-W_2^1, W_3^1 - W_2^1, \dots, W_d^1 - W_2^1)$ in terms of $\Sigma$ as
  $$ \Sigma^{(2)} = 
     \begin{bmatrix}
      \sigma_{22} &  \sigma_{22} \einsfun^\top - \Sigma_{2, \setminus \{1,2\}} \\
      \sigma_{22} \einsfun - \Sigma_{\setminus \{1,2\},2} &  \Sigma_{\setminus \{1,2\}} + \sigma_{22} \einsfun \einsfun^\top - \Sigma_{\setminus \{1,2\},2} \einsfun^\top - \einsfun \Sigma_{2, \setminus \{1,2\}}
    \end{bmatrix}.    
   $$
   It can be checked that the Schur complement of the upper left block $\sigma_{22}$ in the matrix $\Sigma^{(2)}$ is again $S$. Thus, blockwise inversion yields
   $$ \Theta^{(2)} = 
     \begin{bmatrix}
        \sigma_{22}^{-1} + \sigma_{22}^{-2} \left(\sigma_{22} \einsfun^\top - \Sigma_{2,\setminus \{1,2\}}\right)  S^{-1} \left(\sigma_{22} \einsfun - \Sigma_{\setminus \{1,2\},2}\right) & -\sigma_{22}^{-1}\left(\sigma_{22} \einsfun^\top - \Sigma_{2,\setminus \{1,2\}}\right) S^{-1} \\
      -\sigma_{22}^{-1}S^{-1}\left(\sigma_{22} \einsfun - \Sigma_{\setminus \{1,2\},2}\right) &  S^{-1}
    \end{bmatrix}.
     $$
    Comparing these representations of $\Theta^{(1)}$ and $\Theta^{(2)}$ yields the assertion for $i,j \in V \setminus \{1,2\}$.
    For $i\neq 2, j = 2$, we observe
    \begin{align*}
      \sum_{l\neq 2} \Theta^{(2)}_{il} &=  - \sum_{m\neq 1,2}  S^{-1}_{im}  +  \sigma_{22}^{-1}\sum_{m\neq 1,2} S^{-1}_{im}\sigma_{m2} + \sum_{m\neq 1,2}  S^{-1}_{im} =  -\Theta^{(1)}_{i2}.
    \end{align*}
    The case $i,j = 2$ follows similarly.
\end{proof}

\begin{proof}[of Proposition \ref{prop:CI_HR}]
  Let $i,j\in V$ with $i\neq j$ be fixed and choose a $k\neq i,j$. Let $P$ and $\g W$ be as in representation~\eqref{Yk_rep}. Since $Y_k^k = P$ and due to the independence of $P$ and $\g W$ we obtain
  \begin{align*}
    Y_i^k\ci Y_j^k\mid \g Y^k_{\setminus \{i,j\}}  &\iff
    P\times e^{W_i^k - \Gamma_{ik} / 2} \ci P\times e^{W_j^k- \Gamma_{jk} / 2} \mid P, \g W^k_{\setminus \{i,j,k\}}\\
    & \iff W_i^k \ci  W_j^k \mid \g W^k_{\setminus \{i,j,k\}}\\
    &\iff \Theta^{(k)}_{ij}= 0,
  \end{align*}
  where the variable $W_k^k$ can be deleted from the conditioning since it is deterministic given $P$, and therefore the reduced precision matrix $\Theta^{(k)}$ of the vector $\g W_{\setminus k}^k$ appears. The last equivalence follows from the well-known fact that 
  conditional independence in multivariate normal models
  corresponds to zeros in the precision matrix (cf., Example \ref{gauss}).

  Let now $k = i \neq j$ and choose a $k'\notin\{i,j\}$. Lemma \ref{cor_HR} implies
  that
  \begin{align}\label{eq4}
    -\sum_{l\neq k} \Theta^{(k)}_{jl} =  \Theta^{(k')}_{jk}.
  \end{align}
  Since $k'\in V\setminus \{i,j\}$, by Proposition \ref{prop_equiv}, $Y_i \perpe Y_j \mid \g Y_{\setminus \{i,j\}}$ is equivalent to $ Y_k^{k'}\ci Y_j^{k'}\mid \g Y^{k'}_{\setminus \{k,j\}}$. The latter, by the first part of the proof, is then equivalent to $\Theta^{(k')}_{jk}= 0,$ which, together with~\eqref{eq4}, yields the assertion.
  The case $k = j \neq i$ is analogous by symmetry.
\end{proof}

\begin{proof}[of Proposition \ref{mat_compl}]
Let $C_1,\dots,C_m$ be an enumeration of the cliques of the decomposable connected graph $\mathcal G=(V,E)$.
Recall that by assumption, all intersections between pairs of cliques are either empty or contain a single node.
We show how to obtain the unique, $d\times d$-dimensional variogram matrix $\Gamma$ that solves the completion problem \eqref{mat_cons} by adding one clique after the other. We first set
  \begin{align}\label{ind_start}
    \Gamma_{ij} = \Gamma^{(C_1)}_{ij} , \text{ for } i,j\in C_1.
    \end{align}
Let $I_{p-1} = C_1 \cup \dots \cup C_{p-1}$ be the union of the first $p-1$ cliques, $2\leq p \leq m$ cliques that have been chosen in an order such that $\mathcal G$ restricted to $I_{p-1}$ forms a connected graph. Suppose that we have already constructed a unique $|I_{p-1}|\times|I_{p-1}|$-dimensional variogram matrix $\Gamma^{(I_{p-1})}$ that satisfies
  \begin{align}\label{Ip}
    & \begin{cases}
      \Gamma_{ij}^{(I_{p-1})} = \Gamma^{(C_l)}_{ij} , & \text{ for } i,j\in C_l \text{ and all } l=1,\dots, p-1,\\
      \Theta^{(I_{p-1}, k)}_{ij} = 0, & \text{ for all }i,j,k\in I_{p-1}, i,j\neq k \text{ and } (i,j)\notin E,
    \end{cases}
  \end{align}
  where here and in the sequel we use the notation $\Theta^{(J,k)}$ as the inverse of $\Sigma^{(J,k)}= \varphi_{k}(\Gamma^{(J)})$ for a variogram matrix $\Gamma^{(J)}$ on some index set $J \subset V$ and $k\in J$. We next choose a clique, say $C_{p}$, that intersects $I_{p-1}$, and this intersection has to be a single node, say $k_0\in V$.
  Let $I_p = I_{p-1} \cup C_p$ and define the matrix 
\begin{align}\label{theta_block}
  \Theta^{(I_p,k_0)} = 
  \begin{bmatrix}
    \Theta^{(I_{p-1},k_0)} & 0 \\
    0 & \Theta^{(C_p,k_0)}
  \end{bmatrix}.  
  \end{align}
  This matrix is an invertible covariance matrix since its blocks are invertible covariance matrices, and its inverse $\Sigma^{(I_{p},k_0)}$ has the same property with blocks $\Sigma^{(I_{p-1},k_0)}$ and $\Sigma^{(C_p,k_0)}$. This yields an $|I_p| \times |I_p|$-dimensional variogram matrix $\Gamma^{(I_{p})}$ through the mapping $\varphi_{k_0}^{-1}$, which has the form
  \begin{align*}
    \Gamma_{ij}^{(I_{p})} =
    \begin{cases}
      \Gamma^{(I_{p-1})}_{ij} , & \text{ for } i,j\in I_{p-1},\\
      \Gamma^{(C_p)}_{ij} , & \text{ for } i,j\in C_p,\\
      \Gamma_{ik_0}^{(I_{p-1})} + \Gamma_{jk_0}^{(I_{p-1})} &\text{ for } i\in I_{p-1}, j\in C_p \text{  or } j\in I_{p-1}, i\in C_p.
    \end{cases}      
  \end{align*}
  This variogram matrix clearly solves the problem~\eqref{Ip} with $I_{p-1}$ replaced by $I_p$.
  It is unique by construction and the fact that $\varphi_{k_0}$ and $\varphi_{k_0}^{-1}$ are bijections.

  Starting with~\eqref{ind_start} and then adding all cliques for $p=2,\dots, m$ according to the above procedure, we obtain a unique $d\times d$-dimensional variogram $\Gamma = \Gamma^{(I_m)}$ matrix that satisfies all constraints in~\eqref{mat_cons}.
  Comparing with Corollary \ref{cor_HR_dens} it follows that the corresponding density in
 ~\eqref{graph_dens} is $d$-variate H\"usler--Reiss with parameter matrix $\Gamma$.
\end{proof}

\begin{proof}[of Lemma \ref{lem_extr_func}]
  The general formula for extremal functions in Proposition~1 in \cite{dom2016} can be written in terms of the exponent measure density $\lambda$ as 
  \begin{align*}
    \mathbb P(\g U^k \in A) &= \int_{\mathcal E} \einsfun\{\g y / y_k\in A \} \einsfun\{ y_k > 1\} \lambda(\g y) \mathrm d \g y\\
           &= \int_{\mathcal L^k} \einsfun\{\g y / y_k\in A\} f^k(\g y) \mathrm d \g y\\
    &= \mathbb P( \g Y^k/Y^k_k \in A).
  \end{align*}
  Since the density of $\g U^k_{\setminus k} = \g Y^k_{\setminus k}/Y^k_k$ is readily seen to be $\lambda(\g y)$
  for $\g y_{\setminus k} \in [0,\infty)^{d-1}$ and $y_k = 1$, it follows with
  $$ \lambda(\g y) = y_k^{-(d+1)}\lambda(\g y / y_k) = y_k^{-(d+1)} f_{\g U^k_{\setminus k}}(\g y_{\setminus k} / y_k), \quad \g y \in \mathcal E,$$
  that~\eqref{def_ef} is an equivalent definition of extremal functions.

  It follows from Theorem 2 in \cite{dom2016} that for a uniform distribution $T$ 
  on $\{1,\dots, d\}$, the 
  random vector $\g Y^T / \| \g Y^T\|_1$ follows the distribution of the spectral
  measure $H$ on $S_{d-1} = \{\g x \in \mathcal E: \|x \|_1 =1\}$ associated with the max-stable distribution $\g Z$, that is,
  $$\Lambda(A) = d \int_{S_{d-1}} \int_0^\infty u^{-2}  \einsfun\{u \g w \in A \} \mathrm d u H(\mathrm d \g w),\quad A \subset \mathcal E.$$  
  If $A \subset \mathcal L$, then $u\g w \in A$ implies $u \geq 1$, and therefore
  \begin{align*}
    \mathbb P\left(\frac{P\g Y^T}{\| \g Y^T\|_1} \in A \right) &=  \int_{S_{d-1}} \int_1^\infty f_P(u)  \einsfun\{u \g w \in A \} \mathrm d u H(\mathrm d \g w)\\
    &= \frac1d \int_A \lambda(\g y) \mathrm d \g y,
  \end{align*}
  since $f_P(u) = 1/u^{2}, u\geq 1$. For $A = \mathcal L = \mathcal E \setminus [\g 0,\g 1]$ this yields for the conditioning event in~\eqref{Y_rep} 
  \begin{align}\label{rej_prop}
    \mathbb P\left( \frac{ P \|\g Y^T\|_\infty}{\| \g Y^T\|_1} > 1  \right) = \frac{\Lambda(\mathcal L)}{d} = \frac{\Lambda( \g 1)}{d}.
    \end{align}
  Since $\g Y$ has density $\lambda(\g y) / \Lambda(\g 1)$, this concludes the proof.
\end{proof}

\small
  
\bibliographystyle{Chicago}
\bibliography{references}
\end{document}